\documentclass[11pt]{amsart}
\usepackage{geometry}                
\usepackage{graphicx}
\usepackage{amssymb}
\usepackage{epstopdf}
\usepackage{amsmath}
\usepackage{color}
\usepackage{hyperref}
\usepackage{pdfsync}
\usepackage{float}
\usepackage{verbatim}
\usepackage{tikz}
\usetikzlibrary{arrows,decorations.pathmorphing,backgrounds,positioning,fit,petri}
\usetikzlibrary{arrows,patterns,matrix}
\usetikzlibrary{decorations.markings}

\usepackage[all]{xy}
\xyoption{matrix}
\xyoption{arrow}
\usepackage[table]{colortbl}
 
\tikzset{help lines/.style={step=#1cm,very thin, color=gray},
help lines/.default=.5} 
\tikzset{thick grid/.style={step=#1cm,thick, color=gray},
thick grid/.default=1} 

\tikzstyle{ann}=[fill=white, inner sep=1pt, font=\footnotesize{#1}]
\tikzstyle{annfar}=[inner sep=2pt, font=\footnotesize{#1}]
\tikzstyle{annfarer}=[inner sep=3pt, font=\footnotesize{#1}]
\tikzstyle{annrot}=[fill=white, text=blue!75!black, inner sep=1pt, font=\footnotesize{#1}]
\tikzstyle{wall}=[thick]
\tikzstyle{nullwall}=[thick, dotted]
\newcommand{\green}[1]{\textcolor{green!75!black}{#1}}
\newcommand{\red}[1]{\textcolor{red!75!black}{#1}}

\newtheorem{thm}{Theorem}[subsection]
\newtheorem{thm*}{Theorem}
\newtheorem{lem}[thm]{Lemma}
\newtheorem{cor}[thm]{Corollary}
\newtheorem{prop}[thm]{Proposition}

\theoremstyle{definition}
\newtheorem{defn}[thm]{Definition}
\newtheorem{eg}[thm]{Example}

\theoremstyle{remark}
\newtheorem{rem}[thm]{Remark}
 
\numberwithin{equation}{section}

\newcommand{\op}{^{\operatorname{op}}}

\newcommand{\mat}[1]{\ensuremath{
\left[\begin{matrix}#1
\end{matrix}\right]
}}
 
\newcommand{\then}{\Rightarrow}
\newcommand{\ifff}{\Leftrightarrow}

 \newcommand{\kay}{k} 
 \newcommand{\mutindex}{s} 
 \newcommand{\mgs}[1]{\underline{#1}} 
 \newcommand{\aye}{i} 
\newcommand{\jay}{j} 

 \newcommand{\brk}[1]{\langle#1\rangle}

\DeclareMathOperator{\Hom}{Hom}%
\DeclareMathOperator{\Ext}{Ext}%
\DeclareMathOperator{\End}{End}%
\DeclareMathOperator{\Irr}{Irr} 

\DeclareMathOperator{\undim}{\underline{dim}}

\newcommand{\field}[1]{\mathbb{#1}}
\newcommand{\ZZ}{\ensuremath{{\field{Z}}}}

\newcommand{\RR}{\ensuremath{{\field{R}}}}

\newcommand{\commentout}[1]{}

\newcommand{\cC}{\ensuremath{{\mathcal{C}}}}

\newcommand{\cI}{\ensuremath{{\mathcal{I}}}}

\newcommand{\cP}{\ensuremath{{\mathcal{P}}}}

\newcommand{\cV}{\ensuremath{{\mathcal{V}}}}
\newcommand{\cW}{\ensuremath{{\mathcal{W}}}}

\title{Semi-invariant Pictures and two Conjectures on Maximal green sequences}

\author{Thomas Br\"ustle}
\address{D\'epartement de Math\'ematiques, Universit\'e de Sherbrooke, Sherbrooke, CA and Department of Mathematics, Bishops University, Sherbrooke, CA}
\email{thomas.brustle@usherbrooke.ca}
 \thanks{The first author is supported by NSERC and Bishop's University}

\author{Stephen Hermes}
\address{Department of Mathematics, Wellesley College, Wellesley, MA 02481}\email{shermes@wellesley.edu}

\author{Kiyoshi Igusa}
\address{Department of Mathematics, Brandeis University, Waltham, MA 02454}\email{igusa@brandeis.edu}
 \thanks{The third author is supported by NSA Grant \#H98230-13-1-0247}

\author{Gordana Todorov}
\address{Department of Mathematics, Northeastern University, Boston, MA 02115}\email{g.todorov@neu.edu}
\thanks{The fourth author is supported by NSF Grants \#DMS-1103813, \#DMS-0901185}

\date{\today}                      



\begin{document}

\begin{abstract} 
We use semi-invariant pictures to prove two conjectures about maximal green sequences. First: if $Q$ is any acyclic valued quiver with an arrow $\jay\to\aye$ of infinite type then any maximal green sequence for $Q$ must mutate at $\aye$ before mutating at $\jay$. Second: for any quiver $Q'$ obtained by mutating an acyclic valued quiver $Q$ of tame type, there are only finitely many maximal green sequences for $Q'$. Both statements follow from the Rotation Lemma for  reddening sequences and this in turn follows from the Mutation Formula for the semi-invariant picture for $Q$.
\end{abstract}

\maketitle

\section*{Introduction}\label{sec: Introduction}
 
Maximal green sequences are maximal paths in the oriented cluster exchange graph: Fixing an initial seed induces an orientation on the cluster exchange graph, a ``green'' sequence is an oriented path (passing an arrow from source to target is called ``green'', whereas passing an arrow in reverse direction is called ``red''), and a maximal green sequence is one starting from the initial seed (the only source in the oriented exchange graph) to the unique sink (see Figures \ref{fig04} and \ref{fig05}). More generally, any sequence ending in the unique sink of an  oriented cluster exchange graph is called a reddening sequence \cite{Keller}.
Categorification of cluster algebras \cite{BMRRT} has led to a wealth of different interpretations and generalizations of the oriented cluster exchange graph, for instance as poset of functorially finite torsion classes, or of certain t-structures in a triangulated category, see \cite{BY} for an overview. In these poset interpretations, a maximal green sequence is simply a maximal chain. 
Following work of Reineke, Keller studied maximal green sequences to obtain quantum dilogarithm identities \cite{Keller}.
Moreover, maximal green sequences are considered in physics (under the name ``finite chambers'') when studying the BPS spectrum of a quantum field theory with extended supersymmetry, see \cite{ACCERV,X} and references therein.

This paper proves two conjectures about maximal green sequences:

\begin{thm*}[Target before Source Conjecture]
Given an acyclic valued quiver $Q$ with an arrow $ j\xrightarrow{(d_{ji},d_{ij})} i $ of infinite type, \emph{i.e.}, with $d_{ij}d_{ji}\ge4$, any maximal green sequence mutates at the target $i$ before the source $j$.
\end{thm*}

\begin{thm*}[Finiteness Conjecture]
If the valued quiver $Q$ is mutation equivalent to an acyclic quiver of tame type, then $Q$ has only finitely many maximal green sequences.
\end{thm*}

Oriented cluster exchange graphs are associated to cluster algebras; the edges in the cluster exchange graph represent mutations, the fundamental notion in the definition of cluster algebras. The orientation of each edge indicates mutation in the direction of positive $c$-vectors.

The original idea of the proof of the Target before Source Conjecture came from the semi-invariant pictures of \cite{IOTW2}. Using the fact that the lines are labeled by $c$-vectors and the normal orientation on the lines determines the sign of the $c$-vector, green mutations can be visualized as crossing the lines always in the direction of the normal orientation as illustrated in Figure~\ref{fig01}.\\

\begin{figure}[h]
	\begin{center}
	\newdimen\dd
	\dd=6em

\includegraphics{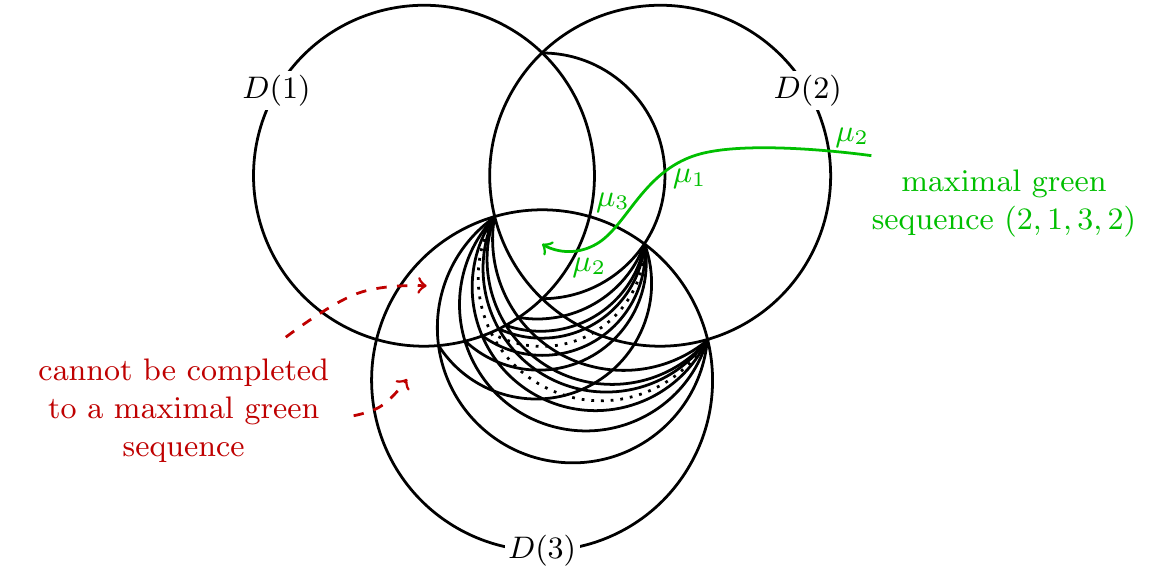}

	\end{center}
	\caption{The semi-invariant picture for the quiver $3\rightrightarrows2\to1$. The dashed red lines cannot be completed to maximal green sequences due to the multiple arrow $3\rightrightarrows2$.}\label{fig01}

\end{figure}

A maximal green sequence is a path going from the outside, unbounded region to the center which only goes inward at each wall. The double arrow $3\rightrightarrows2$ creates infinite families of walls. The solid line is the maximal green sequence $(2,1,3,2)$ and the dotted lines are green sequences which cannot be extended to maximal green sequences showing that maximal green sequences cannot mutate $3$ before $2$. Here the integer $k$ at each step indicates the mutation $\mu_k$ in the direction of the $k$-th $c$-vector, which is the same as mutation at vertex $k$ of the quiver $Q$ defining the cluster algebra.
 
\begin{figure}[h]

\[
	\begin{tikzpicture}
	\draw (0,0) node[right]{$\left(3\rightrightarrows 2\to 1 \right)$};
	\end{tikzpicture}	
	\begin{tikzpicture}
	\draw (0,0) node[right]{$\xrightarrow{\mu_2}$};
	\end{tikzpicture}	 
	\begin{tikzpicture}
	\draw (0,0) node[right]{$\left(3\leftleftarrows 2\leftarrow 1\right)$};
	\begin{scope}[xshift=.15cm]
\draw[<-,thick] (1.7,.2)..controls(1.3,.5)and(.7,.5).. (.3,.2);
\draw[<-,thick] (1.75,.3)..controls(1.3,.6)and(.7,.6).. (.25,.3);
\end{scope}
	\end{tikzpicture}
	\begin{tikzpicture}
	\draw (0,0) node[right]{$\xrightarrow{\mu_1}$};
	\end{tikzpicture}
	\begin{tikzpicture}
	\draw (0,0) node[right]{$\left(3\,\quad\ 2\to 1\right)$};
	\begin{scope}[xshift=.15cm]
\draw[->,thick] (1.7,.2)..controls(1.3,.5)and(.7,.5).. (.3,.2);
\draw[->,thick] (1.75,.3)..controls(1.3,.6)and(.7,.6).. (.25,.3);
\end{scope}
	\end{tikzpicture}
	\begin{tikzpicture}
	\draw (0,0) node[right]{$\xrightarrow{\mu_3}$};
	\end{tikzpicture} 	
	\begin{tikzpicture}
	\draw (0,0) node[right]{$\left(3\,\quad\ 2\to 1\right)$};
	\begin{scope}[xshift=.15cm]
\draw[<-,thick] (1.7,.2)..controls(1.3,.5)and(.7,.5).. (.3,.2);
\draw[<-,thick] (1.75,.3)..controls(1.3,.6)and(.7,.6).. (.25,.3);
\end{scope}
	\end{tikzpicture}
	\begin{tikzpicture}
	\draw (0,0) node[right]{$\xrightarrow{\mu_2}$};
	\end{tikzpicture} 	
	\begin{tikzpicture}
	\draw (0,0) node[right]{$\left(3\,\quad\ 2\leftarrow 1\right)$};
	\begin{scope}[xshift=.15cm]
\draw[<-,thick] (1.7,.2)..controls(1.3,.5)and(.7,.5).. (.3,.2);
\draw[<-,thick] (1.75,.3)..controls(1.3,.6)and(.7,.6).. (.25,.3);
\end{scope}
	\end{tikzpicture}	
\]

\caption{The sequence of quiver mutations for the maximal green sequence $(2,1,3,2)$ on the quiver $3\rightrightarrows2\to 1$. Note that the final quiver is isomorphic to the starting quiver by permuting vertices 1 and 2.}\label{fig01.2}
\end{figure}

The Finiteness Conjecture is known when $Q$ is acyclic \cite{BDP}. The main step in proving both the Finiteness Conjecture and the extension of the known cases to the more general cases of the Finiteness Conjecture is our Rotation Lemma (Theorem \ref{rotation}), which comes from the theory of semi-invariant pictures as developed in \cite{IOTW2}.

For each maximal green sequence there is a permutation $\sigma$ so that, for any arrow $i\to j$ in $Q$ there is an arrow $\sigma^{-1}(i)\to \sigma^{-1}(j)$ in the final quiver $Q'$. For example, in Figure \ref{fig01.2}, the permutation associated to the maximal green sequence $(2,1,3,2)$ is $\sigma=(12)$ making the final quiver $(\sigma^{-1}(3)\rightrightarrows \sigma^{-1}(2)\to\sigma^{-1}(1))=(3\rightrightarrows 1\to 2)$.

\begin{thm*}[Rotation Lemma, Theorem \ref{rotation}]
Let $(\kay_0,\kay_1,\dots,\kay_{m-1})$ be a maximal green sequence for any valued quiver $Q$ with associated permutation $\sigma$. Then the sequence
\[
	(\kay_1,\kay_2,\dots,\kay_{m-1},\sigma^{-1}(\kay_0))
\]
is a maximal green sequence on $\mu_{\kay_0}Q$ with the same permutation $\sigma$. More generally, if the original sequence is a reddening sequence (see Definition \ref{def:reddening}), then the rotated sequence has the same number of red mutations as the original sequence.
\end{thm*}

The Rotation Lemma, together with Lemma \ref{lemmaB} and Corollary \ref{cor:general conjecture 1} imply the Target before Source Conjecture as stated in Corollary \ref{cor:ibeforej}: If there were a maximal green sequence which mutates $j$ before $i$, then there would be a rotation of that sequence which would produce a quiver $Q'$ (not necessarily acyclic) and a maximal green sequence for $Q'$ which has $\mu_j$ as the first mutation. By Lemma \ref{lemmaB} there would still be an arrow $j\to i$ in $Q'$ of infinite type and, by Corollary \ref{cor:general conjecture 1}, this is impossible in a maximal green sequence.

Note that the Target before Source Conjecture is not true if the word ``acyclic'' is removed. However, the Rotation Lemma which holds for any valued quiver, not necessarily acyclic, together with Theorem \ref{thm:ibeforej} implies the following version of the conjecture for general valued quivers. This is based on a suggestion made to us by Greg Muller.

\begin{thm*}[Corollary \ref{cor:general conjecture 1}] Consider any maximal green sequence on any valued quiver $Q$. Then, at each step, the mutation is at a vertex of the mutated quiver $Q'$ which is not the source of any arrow of infinite type.
\end{thm*}

\newsavebox{\quivvv}
\savebox{\quivvv}{\begin{minipage}[h]{0.13\linewidth}\begin{tikzpicture}
	\draw (0,0) node[right]{$3\leftarrow 2\leftarrow 1$};
	\begin{scope}
\draw[<-,thick] (1.7,.2)..controls(1.3,.5)and(.7,.5).. (.3,.2);
\draw[<-,thick] (1.75,.3)..controls(1.3,.6)and(.7,.6).. (.25,.3);
\end{scope}
	\end{tikzpicture}\end{minipage}}

\begin{figure}[h]
	\begin{center}
	\newdimen\dd
	\dd=6em

\includegraphics{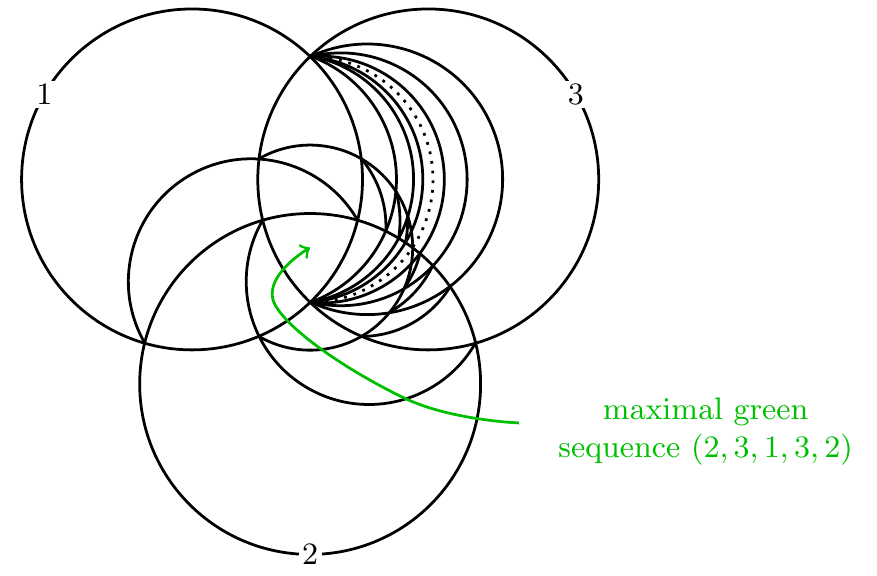}

	\end{center}
	\caption{The Target before Source Conjecture is not true if $Q$ is not acyclic. For the quiver \usebox{\quivvv} there is a double arrow $3\rightrightarrows1$ and a maximal green sequence $(2,3,1,3,2)$ which mutates 3 before 1.}\label{fig02}
\end{figure}

To prove the Finiteness Conjecture, we need to strengthen the theorem of \cite{BDP} to the following theorem.

\begin{thm*}[Theorem \ref{finiteness}]
Let $Q$ be any acyclic tame quiver and let $r\ge0$. Then there are at most finitely many reddening sequences on $Q$ with at most $r$ red mutations.
\end{thm*}

To see that this theorem implies the Finiteness Conjecture, choose a fixed mutation sequence $(j_1,\dots,j_r)$ from the acyclic tame quiver $Q$ to a mutation equivalent quiver $Q'$ which might not be acyclic. For each maximal green sequence $(\kay_0,\dots,\kay_{m-1})$ on $Q'$ we associate a reddening sequence $(j_r,\dots, j_1,j_1,\dots,j_r,\kay_0,\dots,\kay_{m-1})$ on $Q'$ with exactly $r$ red mutations. Then, by applying the Rotation Lemma $r$ times, we see that
\[
	(j_1,\dots,j_r,\kay_0,\dots,\kay_{m-1},\sigma^{-1}(j_r),\dots,\sigma^{-1}(j_1))
\]
is a reddening sequence on $Q$ with exactly $r$ red mutations. By the theorem, there are at most finitely many such sequences. Therefore, $Q'$ has at most finitely many maximal green sequences.

A motivation for the Target before Source and Finiteness Conjectures comes from the physics literature: In \cite{ACCERV} an explicit method is given to construct maximal green sequences for all quivers stemming from a triangulation of a surface. They cut the surface into manageable pieces and then glue the quiver together according to a rule that implicitly uses the Target before Source Conjecture: the setup is such that automatically the maximal green sequence would first mutate the target of a multiple arrow.
In Section 4.2 of \cite{X}, the Target before Source Conjecture is explicitly formulated by saying if there are two green vertices connected by double arrows, then one cannot mutate the source in order to obtain a finite chamber. The Rotation Lemma is also suggested by Figure 11 in \cite{X}.


\section{Preliminaries}\label{sec: Preliminaries}

The context of this work is that of cluster algebras. Later on, we will restrict to the acyclic case. We will use the following notation and definitions throughout the paper. 

An $n\times n$ integer matrix $B$ is called \emph{skew-symmetrizable} if there is a diagonal matrix $D$ with positive integer diagonal entries so that $DB$ is skew-symmetric. We recall the notion of matrix mutation, given in \cite{FZ} as an ingredient of the definition of cluster algebras:

\begin{defn} For any $m \times n$ matrix $\widetilde B=(b_{ij})$ with $m \ge n$, and any $1\le k\le n$, the \emph{mutation $\mu_k \widetilde B$ of $\widetilde B$ in the $k$-direction} is defined to be the matrix $\widetilde B'=(b_{ij}')$ given by
\begin{equation}\label{eq: mutation of B}
	b_{ij}'=\begin{cases}  -b_{ij}& \text{if } i=k \text{ or } j=k\\
   b_{ij}+b_{ik}|b_{kj}| & \text{if } b_{ik}b_{kj}>0\\
  b_{ij} & \text{otherwise}
    \end{cases}
\end{equation}
For any finite sequence of positive integers $1 \le k_1,k_2,\cdots,k_r\le n$ we have the \emph{iterated mutation} $\mu_{k_r}\cdots \mu_{k_1}\widetilde B$ of $\widetilde B$.
\end{defn}

A \emph{valued quiver} is a (not necessarily acyclic) quiver $Q$ with set of vertices $Q_0$ and set of arrows $Q_1$ endowed with positive integer valuations on arrows as $ i\xrightarrow{(d_{ij},d_{ji})} j $ and a positive integer valuation $f_i $ for each vertex $i\in Q_0$ satisfying $d_{ij}f_j=d_{ji}f_i$.
The associated Euler matrix $E=(E_{ij})$ of this quiver is given by $E_{ii} = f_i$ for all $i\in Q_0$ and $E_{ij} = - d_{ij}$ for each arrow $ i\xrightarrow{(d_{ij},d_{ji})} j $ in $Q_1$.
The associated skew-symmetrizable matrix $B$ is determined by the equation $DB=E^t\!-\!E$ where $D$ is the diagonal matrix with entries $f_i$ on the diagonal. Conversely, given $B$ and $D$, the valued quiver $Q$ is given as follows. There is an arrow $i\to j$ in $Q$ whenever $b_{ij}>0$ with valuation $(-b_{ji},b_{ij})$. The integer valuations $f_i$ are the diagonal entries of $D$. 

We refer to \cite{DR} or \cite{IOTW2} for more details on valued quivers, and their relation to modulated quivers and their representations in the acyclic case. For the general case, see \emph{e.g.}, \cite{N}.

\begin{eg} \label{example}
We use the following example of a valued quiver $Q$ to illustrate some of the definitions:
\[
3\xrightarrow{(3,1)} 2 \xrightarrow{(1,1)} 1 \text { and } \ \  f_3=3,\  f_2=1,\  f_1=1.
\]

The Euler matrix $E$, the skew-symmetrizable matrix $B$ and the diagonal matrix $D$ in this case will be:

\[
E=\left[\begin{matrix}1&0&0\\
-1&1&0\\
0&-3&3
\end{matrix}\right] \qquad 
B=\left[\begin{matrix}0&-1&0\\
1&0&-3\\
0&1&0
\end{matrix}\right]
\qquad D=\left[\begin{matrix}1&0&0\\
0&1&0\\
0&0&3
\end{matrix}\right].
\]
\end{eg}

The construction of the skew-symmetrizable matrix $B$ associated to a valued quiver $Q$ deletes all loops and oriented two-cycles in $Q$, so we will assume that $Q$ has no oriented cycles of length one or two. In fact, the construction above induces a bijection between skew-symmetrizable integer matrices and valued quivers  without loops and oriented two-cycles. Mutation of the matrices $B$ corresponds to quiver mutation.

We recall from \cite{FZ} that the skew-symmetrizable matrix $B$ determines a family of matrices called \emph{$c$-matrices} whose columns are called \emph{$c$-vectors}: If $B$ is of size $n\times n$ we form the $(2n)\times n$ extended matrix $\widetilde{B}=\mat{B\\I_n}$ by appending the $n\times n$ identity matrix $I_n$ to the bottom of $B$. A $c$-matrix is then any $n\times n$ matrix $C$ appearing as the bottom half of some $\widetilde{B}'=\mat {B'\\ C}$ obtained from $\widetilde{B}$ by successive mutations. If we denote by $Q'$ the quiver obtained from $Q$ by the same sequence of mutations, that is, $Q'$ is the valued quiver corresponding to $B'$, we may call the matrix $C$ the $c$-matrix of $Q'$. The corresponding $c$-vectors are naturally in bijection with the vertices of $Q'$. It should be noted that strictly speaking, the $c$-matrices and $c$-vectors depend on the sequence of mutations taken to arrive at $Q'$, and not directly on $Q'$ itself, see  \cite{FZ} or  \cite{NZ} for details. This paper is concerned precisely with sequences of mutations and the associated $c$-matrices and $c$-vectors.

A vertex of $Q'$ is \emph{green} (resp. \emph{red}) if all entries of the corresponding $c$-vector are nonnegative (resp. nonpositive). A mutation sequence $(k_0,k_1,\dots, k_{m-1})$ is a \emph{maximal green sequence} if each vertex $k_s$ is green in $Q_s=\mu_{k_{s-1}}\cdots\mu_{k_0}Q$ and moreover all vertices of $Q'=Q_{m}$ are red.

\begin{rem}
A skew-symmetrizable integer matrix is called sign coherent if the entries of any $c$-vector are either all nonnegative or all nonpositive, \emph{i.e.}, are either green or red. Generalizing results obtained earlier for skew-symmetric matrices \cite{DWZ}, it has recently been shown in \cite{GHKK} that every skew-symmetrizable integer matrix is sign coherent. A proof for acyclic valued quivers is also given in \cite{IOTW2}.
Thus, any vertex of the quiver  $Q'$ above is either green or red. 

It should be noted that the determinant of the $c$-matrix is always $\pm1$ since the mutation $\mu_k$ changes the sign of the $k$-th column and adds integer multiples of this column to certain other columns. Also, by \cite{NZ}, the $c$-matrix $C$ determines $B'$ by the equation: $DB'=C^tDBC$.
\end{rem}

Here we work with the more general notion of reddening sequences introduced in \cite{Keller}.

\begin{defn}\label{def:reddening}
A \emph{reddening sequence} is a mutation sequence $(k_0,k_1,\dots, k_{m-1})$ in which all vertices of the final quiver $Q_{m}$ are red. It will sometimes be convenient to call a reddening sequence with exactly $r$ mutations at red vertices an \emph{$r$-reddening sequence}. In particular, a $0$-reddening sequence is just a maximal green sequence.
\end{defn}


\section{Rotation Lemma for Reddening Sequences}

In order to prove the Rotation Lemma we need  precise formulas which relate the sequence of $c$-matrices obtained by a mutation sequence $(\kay_0,\kay_1,\dots,\kay_m)$ of a skew-symmetrizable matrix $B$ and the sequence of $c$-matrices obtained by the mutation sequence $(\kay_1,\kay_2,\dots,\kay_m)$ of the skew-symmetrizable matrix $\mu_{\kay_0}B$. In both cases the initial $c$-matrix is the identity matrix $I=I_n$.

\subsection{Mutation formula}
In Theorem \ref{formula} we give a formula which relates $c$-vectors of a skew-symmetrizable matrix and its once-mutated matrix. This formula corresponds to the identity given in proposition 1.4  of \cite{NZ}, but we give here a different version of computing the signs. Our proof uses the concept of  $k$-hemispheres which we will also need later on. We also use the notion of $g$-matrices and  sign-consistency of $g$-vectors. For our purpose it is not necessary to introduce the definition of a  $g$-matrix, we rather use the result of Nakanishi and Zelevinsky in \cite{NZ} that the $g$-matrix $G$, which by its original definition \cite{FZ} is an integer matrix, is related to the $c$-matrix $C$ as $G=(DC^{-1}D^{-1})^t$. The sign-coherence of $c$-vectors implies by \cite{NZ} that $g$-vectors are sign-consistent, \emph{i.e.}, the rows of the $g$-matrix $G$ are sign coherent. 
Of particular importance is the case when a $c$-vector is a simple root, that is of the form $\pm e_k$ where $e_k$ denotes the $k$-th standard vector.
We show in  the following lemma that the occurrence of a simple root implies some condition on the valuations $f_i$ of the vertices.

\begin{lem}\label{lem: CD=DC}
If the $j$-th column of a $c$-matrix $C$ is $\pm e_k$ then $f_k=f_j$. In particular, if $C$ is a permutation matrix or negative permutation matrix then $CD=DC$.
\end{lem}

\begin{proof}
Let $G$ be the corresponding $g$-matrix. Then $G$ is an integer matrix with determinant $\pm1$ and $(G^t)^{-1}=DCD^{-1}$. The $j$-th column of the latter matrix is $\pm f_kf_j^{-1}e_k$. The coefficient $\pm f_kf_j^{-1}$ is an integer which divides $\det C=\pm1$. Therefore, $f_k=f_j$.
\end{proof}

Let $B=B_0$ be a skew-symmetrizable matrix and let $\widetilde B_0=\mat{B_0\\I_n}$ be the extended matrix with initial $c$-matrix $C_0=I=I_n$. Consider the two sequences of mutations, both closely related to the matrix $B$:

\begin{equation}\label{ast}
\begin{aligned}
\widetilde B_0=\mat{B_0\\I}\xrightarrow{\mu_{\kay_0}}\mat{B_1\\C_1}\xrightarrow{\mu_{\kay_1}} &\mat{B_2\\C_2} \xrightarrow{\mu_{\kay_2}} \dots
\xrightarrow{\mu_{\kay_{m-1}}} \mat{B_m\\C_m},  \text { and} \\
\widetilde B'_1=\mat{B_1\\I}\xrightarrow{\mu_{\kay_1}} &\mat{B_2\\C'_2}\xrightarrow{\mu_{\kay_2}}
 \dots
\xrightarrow{\mu_{\kay_{m-1}}} \mat{B_m\\C'_m}.
\end{aligned}
\end{equation}

For each $\mutindex\geq 1$ we will express the $c$-matrix $C'_\mutindex$ in terms of the $c$-matrix $C_\mutindex$. For this it is convenient to write matrix mutation in terms of column operations given by multiplication with the following matrices $X^+_j$, $X^-_j$.

\begin{defn}\label{Xmatrix}
Let $B$ be an $n\times n$ skew-symmetrizable matrix. For $\varepsilon=\pm$ define $X_j^\varepsilon$ to be the matrix equal to the identity matrix $I_n$ except for its $j$-th row which is given by
 \[
 	(X_j^\varepsilon)_{jk}=\begin{cases} -1 & \text{if } k=j\\
   \max(\varepsilon b_{jk},0) & \text{if } k\neq j
    \end{cases}
 \]
\end{defn}

\begin{eg} For the quiver $ 3\xrightarrow{(3,1)} 2 \xrightarrow{(1,1)} 1 \text { with } f_3\!=\!3,\  f_2\!=\!1,\  f_1\!=\!1$ as in Example \ref{example} with matrix $B$, the matrices $X_2^+$ and  $X_2^-$ are:
\[
B=\left[\begin{matrix}0&-1&0\\
1&0&-3\\
0&1&0
\end{matrix}\right]
\qquad
X_2^+=\left[\begin{matrix}1&0&0\\
1&-1&0\\
0&0&1
\end{matrix}\right]
\qquad 
X_2^-=\left[\begin{matrix}1&0&0\\
0&-1&3\\
0&0&1
\end{matrix}\right].
\qquad 
\]
 \end{eg}
 
 We state without proof some of the basic properties of these matrices.  
 
\begin{lem} \label{XX} Let $B$ be an $n\times n$ skew-symmetrizable matrix. Then:
\begin{enumerate}
\item $X_j^+X_j^+=I_n=X_j^-X_j^-$ for $j\in\{1,2,\dots,n\}$, that is, mutation is an involution,
\item $X_j^+X_j^-=I_n+J_jB$,  where $J_j$ is the diagonal matrix with $d_{jj}=1$ and $d_{ii}=0$ for  all $i\neq j$.\qed
\end{enumerate}
\end{lem}

\begin{lem}\label{mujC}
 Let $B$ be a skew-symmetrizable matrix,  $\widetilde B=\mat{B\\C}$ and 
 $j\in\{1,\dots,n\}$. Let $\mu_jB$  and $\mu_jC$ be the mutated matrices $B$ and $C$ in the direction of the $j$-th column. Then: 
\begin{enumerate}
\item $\mu_jC= C X_j^+$ if the $j$-th column of the matrix $C$ has entries $\geq 0$ and \\$\mu_jC= C X_j^-$ if the $j$-th column of the matrix $C$ has entries $\leq 0$.
\item $\mu_jI=  X_j^+$  for $j\in\{1,\dots,n\}$.
\item The $g$-matrices $G,\mu_jG$ corresponding to $C,\mu_j G$ differ only in their $j$-th columns.
\end{enumerate}
\end{lem}
\begin{proof}
(1) follows directly from the definition of matrix mutation. It is also formulated in proposition 1.3 of \cite{NZ}. (2) follows from (1). (3) follows from (1) since
\[
	D^{-1}(\mu_jG)^tD=(\mu_jC)^{-1}=(CX_j^{\varepsilon})^{-1}=(X_j^\varepsilon)^{-1} C^{-1}=X_j^\varepsilon C^{-1}
\]
which differs from $C^{-1}=D^{-1}G^tD$ only in its $j$-th row.
\end{proof}

Sign consistency of $g$-vectors means that, for any fixed $\widetilde B_s$, the $\kay$-th coordinate of the corresponding $g$-vectors all have the same sign, \emph{i.e.} the $k$-th row of the $g$-matrix $G$ is sign coherent for each $k\in\{1,\dots,n\}$.
We use this fact in the following definition.

\begin{defn}
Let $k\in \{1,2,\dots, n\}$. Define $H_{k}^+$ to be the set of all $c$-matrices $C$ so that the corresponding $g$-vectors have $k$-th coordinate $\ge0$. Let $H_{k}^-$ be the other $c$-matrices (whose $g$-vectors have nonpositive $k$-th coordinates). We call $H_{k}^+,H_{k}^-$ the \emph{$k$-hemispheres}.
\end{defn}

Of course one could define in the same way hemispheres using the $c$-vectors, the difference lies in the index: if $C$ has its $j$-th column $\ge 0$, then the corresponding $g$-matrix $G$ has its $k$-th row $\ge 0$, for some index $k$. It turns out to be more convenient following the index $k$. 
A case of particular importance is when the $j$-th column of $C$ is $\pm e_k$, as discussed in Lemma \ref{lem: CD=DC}. 
The following  lemma gives a precise description when a $c$-matrix can change hemispheres through mutation, it turns out that this is possible only when mutating at a simple root. It applies to any $k\in \{1,2,\dots, n\}$, but we mainly use it for $k=\kay_0$, \emph{i.e.}, for the $c$-vector of the first mutation in the  sequence of mutations (\ref{ast}).

\begin{lem}\label{hemispheres} Let $C$ be a $c$-matrix and $G$ the corresponding $g$-matrix.
\begin{enumerate}
\item If any $c$-vector in the matrix $C$ is $e_{k}$ (resp. $-e_k$) then the $k$-th row in $G$ has all entries $\geq 0$ (resp. $\le0$), hence the matrix $C$ is in $H_{k}^+$ (resp. $H_k^-$).
\item The matrices $C$ and $\mu_jC$ are in different $k$-hemispheres if and only if the $j$-th vector in $C$ is $\pm e_{k}$.
\end{enumerate}
\end{lem}

\begin{proof} (1) Suppose that the $j$-th column of the matrix $C$ is $\pm e_k$. Using the equation $G^tDC=D$ from \cite{NZ}, we see that the $(k,j)$ entry of the matrix $G$ must be $\pm1$. Therefore, by sign-consistency of $g$-vectors, the $k$-th row of $G$ must have the same sign. This proves (1). \\
(2) Furthermore, the $j$-th column of $\mu_jC$ will be $\mp e_k$ and the sign of the $k$-th row of $\mu_jG$ will be the opposite of that of $G$. This proves the implication $(\Leftarrow)$. Conversely, suppose that $C,\mu_jC$ are in opposite $k$-hemispheres. Then the $k$-th rows of $G,\mu_jG$ have opposite sign. By Lemma \ref{mujC}(3), $G,\mu_jG$ differ only in their $j$-th columns. So, the $k$-th rows of $G$ and $\mu_jG$ can have only one nonzero entry in position $(k,j)$. Then the $j$-th row of $G^{-1}$ has only one nonzero entry in position $(j,k)$. So, $C=(DG^{-1}D^{-1})^t$ has only one nonzero entry in its $j$-th column in position $(k,j)$, \emph{i.e.}, the $j$-th column of $C$ is $\pm e_k$.
\end{proof}

\begin{thm} [Mutation formula]  \label{formula} Let $B=B_0$ be a skew-symmetrizable matrix. Consider the two sequences of mutations in \eqref{ast}.

Then for all $\mutindex\ge0$ we have: 

\[
C'_\mutindex=X_{\kay_0}^{\varepsilon(\mutindex)}C_\mutindex \qquad\qquad \text{ where }  \qquad\qquad 
{\varepsilon(\mutindex)=\begin{cases}+ & \text{if } C_\mutindex\in H_{\kay_0}^-\\
   - & \text{if } C_\mutindex\in H_{\kay_0}^+
   \end{cases}}.
\]
  
\end{thm}

\begin{rem}
Proposition 1.4  of \cite{NZ} gives a formula relating the two matrices $C'_\mutindex$ and $C_\mutindex$ by multiplication with a matrix, however this matrix depends on a sign that has to be traced back along the sequence of mutations. We give a formula that allows one to compute the sign using the $\kay_0$-hemisphere. This formula can be derived from the mutation formula for $g$-matrices in \cite{GHKK,R13}. We reprove it in a form more convenient for our purposes.
\end{rem}

\begin{proof}
The proof will be by induction on $s$ for the following two statements:
\begin{enumerate}
\item[$(a_{\mutindex})$] For each $k$ let $c_k,c_k'$ denote the $k$-th columns of the matrices $C_\mutindex$ and $C'_\mutindex$. Then $c_k,c_k'$ have the same sign unless $c_k=\pm e_{\kay_0}$ in which case $c_k'=-c_k$.\\
\item[$(b_{\mutindex})$] $C'_\mutindex=X_{\kay_0}^{\varepsilon (\mutindex)}C_\mutindex$ where $\varepsilon(\mutindex)$ is as defined above.\\
\end{enumerate}

$(\mutindex=1)$ $(a_1)$ By definition $C_1=\mu_{\kay_0}C_0=\mu_{\kay_0}I_n=X_{\kay_0}^+$ by Lemma \ref{mujC} (2). 
 Notice that the matrix $X_{\kay_0}^+$ has $-1$ in the $(\kay_0,\kay_0)$ entry, positive diagonal entries and only positive entries stemming from the $\kay_0$-th row of the matrix $B$, hence all but the $\kay_0$-th columns are positive, while the $\kay_0$-th column is  $-e_{\kay_0}$.
 
The matrix $C' _1$ equals $I_n$, hence all columns are positive. Therefore $(a_1)$ holds.\\
$(b_{1})$  $C'_1=I_n=X_{\kay_0}^+X_{\kay_0}^+=X_{\kay_0}^+C_1$ follows from Lemma \ref{XX} (1) and Lemma \ref{mujC} (2).
  
Assume $(a_{\mutindex})$ and $(b_{\mutindex})$ hold. 
  
\underline{Claim}: $(b_{\mutindex+1})$ holds, \emph{i.e.} $C'_{\mutindex+1}=X^{\varepsilon(\mutindex+1)}_{\kay_0}C_{\mutindex+1}$. \\
Proof of the claim: We know that $C'_\mutindex=X^{\varepsilon(\mutindex)}_{\kay_0}C_\mutindex$. Consider $C'_{\mutindex+1}=\mu_{\kay_{\mutindex}}C'_\mutindex$ and $C_{\mutindex+1}=\mu_{\kay_{\mutindex}}C_\mutindex$.\\

\underline{Case 1}: The $\kay_\mutindex$-th columns in both $C_\mutindex$ and $C'_\mutindex$ are positive. Then, by Lemma \ref{mujC} (1) it follows that $C'_{\mutindex+1}=\mu_{\kay_\mutindex}C'_\mutindex=C'_\mutindex X^+_{\kay_\mutindex}$ and $C_{\mutindex+1}=\mu_{\kay_\mutindex}C_\mutindex=C_\mutindex X^+_{\kay_\mutindex}$. Therefore $C'_{\mutindex+1}=\mu_{\kay_{\mutindex}}C'_\mutindex=C'_\mutindex X^+_{\kay_\mutindex}= X^{\varepsilon(\mutindex)}_{\kay_0}C_\mutindex X^+_{\kay_\mutindex}= X^{\varepsilon(\mutindex)}_{\kay_0}C_{\mutindex+1}$. We only need to show that $\varepsilon(\mutindex)=\varepsilon(\mutindex+1)$.
To see that, we notice that by the induction hypothesis and assumption that they have the same sign, the $\kay_\mutindex$-th columns of $C_\mutindex$ and $C'_\mutindex$ are not $\pm e_{\kay_0}$. Therefore, by Lemma \ref{hemispheres}, $C_{\mutindex+1}$ and $C_\mutindex$ are in the same $\kay_0$-hemisphere. Hence $\varepsilon(\mutindex+1)=\varepsilon(\mutindex)$  by the definition of $\varepsilon$. Therefore $C'_{\mutindex+1}=X^{\varepsilon(\mutindex+1)}_{\kay_0}C_{\mutindex+1}$.\\

\underline{Case 2}: The $\kay_\mutindex$-th columns in both $C_\mutindex$ and $C'_\mutindex$ are negative. The proof is the same as in case 1, using the matrix $X^-_{\kay_\mutindex}$.\\

\underline{Case 3}: The $\kay_\mutindex$-th columns in  $C_\mutindex$ and $C'_\mutindex$ have opposite signs. By induction hypothesis $(a_\mutindex)$, this means that the $\kay_\mutindex$-th columns of $C'_\mutindex$ and $C_\mutindex$ are $e_{\kay_0}$ and $-e_{\kay_0}$ (or conversely  $-e_{\kay_0}$ and $e_{\kay_0}$). In matrix form this condition is:
   
    \[
    	C_\mutindex J_{\kay_\mutindex}=\pm J_{\kay_0}P_\tau
    \]
where $J_{\kay_\mutindex}$ is as in Lemma \ref{XX} and $P_\tau$ is the permutation matrix of the transposition $\tau=(\kay_0,\kay_\mutindex)$. ($P_\tau D=DP_\tau$ by Lemma \ref{lem: CD=DC}.) Since $e_{\kay_0}$ has entries $\geq 0$  and $-e_{\kay_0}$ has entries $\leq 0$ it follows by Lemma \ref{hemispheres} that:  
\[
C_\mutindex\in H^+_{\kay_0}, \ \ \ \ C'_\mutindex\in H^-_{\kay_0}, \ \ \ \ C_{\mutindex+1}=\mu_{\kay_\mutindex}C_\mutindex\in H^-_{\kay_0}, \ \ \ \ C'_{\mutindex+1}=\mu_{\kay_\mutindex}C'_\mutindex\in H^+_{\kay_0}.
\]
Applying now Lemma \ref{mujC}, it follows that:
\[
C'_{\mutindex+1}=\mu_{\kay_{\mutindex}}C'_\mutindex=C'_{\mutindex}X^-_{\kay_\mutindex}\quad \text{ and  } \quad C_{\mutindex+1}=\mu_{\kay_{\mutindex}}C_\mutindex=C_{\mutindex}X^+_{\kay_\mutindex}.
\]
By the induction hypothesis $(b_\mutindex)$ and the fact  that $C_\mutindex\in H^+_{\kay_0}$ it follows that 
\[
C'_\mutindex=X^{\varepsilon(\mutindex)}_{\kay_0}C_\mutindex= X^-_{\kay_0}C_ j\quad \text{ and therefore } \quad C'_{\mutindex+1}=C'_{\mutindex}X^-_{\kay_\mutindex}= X_{\kay_0}^{-}C_\mutindex X^-_{\kay_\mutindex}.\]
In order to prove $(b_{\mutindex+1})$ in this case and using the fact that  $C_{\mutindex+1}\in H^-_{\kay_0}$, we want to show: 
\[
C'_{\mutindex+1}=X^{\varepsilon(\mutindex+1)}_{\kay_0}C_{\mutindex+1}=X^+_{\kay_0}C_{\mutindex+1}.
\]

From the above formulas, it is enough to show:
\[
X_{\kay_0}^{-}C_\mutindex X^-_{\kay_\mutindex}=X^+_{\kay_0}C_{\mutindex+1}=X^+_{\kay_0}C_{\mutindex}X^+_{\kay_\mutindex}.
\]
Using Lemma \ref{XX} (1) and (2), it will be enough to show:
\[
X^+_{\kay_0}X_{\kay_0}^{-}C_\mutindex=C_{\mutindex}X^+_{\kay_\mutindex}X^-_{\kay_\mutindex} \ \ \text{ or, equivalently, }\ \ \ (I_n+J_{\kay_0}B_0)C_\mutindex=C_{\mutindex}(I_n+J_{\kay_\mutindex}B_\mutindex).
\]
Since $I_nC_\mutindex=C_\mutindex I_n$, it suffices to show that $J_{\kay_0}B_0C_\mutindex=C_{\mutindex}J_{\kay_\mutindex}B_\mutindex$. But, we have $C_\mutindex J_{\kay_\mutindex}=J_{\kay_0}P_\tau$. So,
\[
	C_\mutindex J_{\kay_\mutindex}D^{-1}C_\mutindex^tD=C_\mutindex J_{\kay_\mutindex}D^{-1}J_{\kay_\mutindex}C_\mutindex^tD=J_{\kay_0}P_\tau D^{-1}P_\tau J_{\kay_0}D=J_{\kay_0} D^{-1} J_{\kay_0}D=J_{\kay_0}.
\]
Multiplying both sides by $B_0C_\mutindex$ and using the equation $B_\mutindex=D^{-1}C_\mutindex^tDB_0C_\mutindex$ from \cite{NZ} we get:
\[
	C_\mutindex J_{\kay_\mutindex}B_\mutindex=C_\mutindex J_{\kay_\mutindex}D^{-1}C_\mutindex^tDB_0C_\mutindex=J_{\kay_0}B_0C_\mutindex
\]
as required. That was the last step which now implies $(b_{\mutindex+1})$, \emph{i.e.}, $C'_{\mutindex+1}=X^{\varepsilon(\mutindex+1)}_{\kay_0}C_{\mutindex+1}.$

\underline{Claim}: Assuming $(a_{\mutindex})$ and $(b_{\mutindex})$ the statement $(a_{\mutindex+1})$ holds. \\
The proof of this claim uses the fact that $(b_{\mutindex+1})$ holds from above, \emph{i.e.}  $C'_{\mutindex+1}=X^{\varepsilon(\mutindex+1)}_{\kay_0}C_{\mutindex+1}.$ The matrix $X^{\varepsilon(\mutindex+1)}_{\kay_0}$ has nonnegative entries in all rows different from $\kay_0$-th. 
Hence each row different from $\kay_0$-th row  in the matrices  $C'_{\mutindex+1}$ and $C_{\mutindex+1}$ will have the same sign. Because of sign coherence of $c$-vectors, this means that each pair of corresponding columns in $C'_{\mutindex+1}$ and $C_{\mutindex+1}$ which are $\neq \pm e_{\kay_0}$ must have the same sign. 
If the $k$-th column vector in  $C_{\mutindex+1}$  is $\pm e_{\kay_0}$ then the $k$-th column vector in $C'_{\mutindex+1}$ will be $ \mp e_{\kay_0}$ since the $(\kay_0,\kay_0)$ entry in $X^{\varepsilon(\mutindex+1)}_{\kay_0}$ is $=-1$.  Therefore $(a_{\mutindex+1})$ holds. 
\end{proof}

\subsection{Rotation Lemma}

To state the Rotation Lemma we need the following.

\begin{lem}\label{all red is -Ps}
If all entries of a $c$-matrix $C$ are non-positive then $C$ is a negative permutation matrix. In particular, any reddening sequence ends in a negative permutation matrix $-P_\sigma$, whose $j$-th column is $-e_{\sigma(j)}$, for some permutation $\sigma$.
\end{lem}

\begin{defn}
We call $\sigma$ the \emph{permutation} associated to the reddening sequence.
\end{defn}

\begin{proof}[Proof of Lemma \ref{all red is -Ps}]
We first note that $C$ is all negative (\emph{i.e.}, all nonzero entries are negative) if and only if the corresponding $g$-matrix $G$ is all negative. This follows from the equation $DCD^{-1}G^t=I_n$. For any $k$, the product of the $k$-th row of $DCD^{-1}$ with the $k$-th column of $G^t$ is 1. So, one of the entries in the $k$-th column of $G^t$ must be negative. By sign consistency of $g$-vectors, all nonzero entries in the $k$-th column of $G^t$ are negative. Since this holds for all $k$, all nonzero entries of $G$ are negative. The converse follows from the equation $D^{-1}G^tDC=I_n$.

Now suppose that $C$ is all negative. Then, $C$ lies in $H_k^-$ for all $k$. For each $j$, $\mu_jC$ has positive $j$-th column. Therefore, $\mu_jG$ must also have a positive entry, say in the $k$-th row. Then $\mu_jC$ lies in $H_k^+$ which implies, by Lemma \ref{hemispheres}(2), that the $j$-th column of $C$ must be $-e_k$. Since this holds for all $j$, $C$ is a negative permutation matrix.
\end{proof}

We will show that the Rotation Lemma follows from the Mutation Formula in Theorem \ref{formula} and the following lemma.

\begin{lem}\label{one-more-time}
Let $B$ be a skew-symmetrizable matrix. For any $1\le k\le n$ every reddening sequence for $B$ will mutate the $c$-vector $+e_k$ one more time than it mutates the $c$-vector $-e_k$.
\end{lem}

\begin{proof}
Consider the sequence of $c$-matrices $I_n=C_0,\dots,C_m=-P_\sigma$ of a reddening sequence on $B$. The first $c$-matrix $C_0=I_n$ lies in $H_k^+$ for every $k$ and the last $c$-matrix lies in $H_k^-$ for every $k$. So, during the mutation sequence it must pass from the positive to the negative side of the hyperplane $H_k$ one more time than it goes from the negative to the positive side. By Lemma \ref{hemispheres}(2), the first event occurs when the mutated $c$-vector is $+e_k$. The second event occurs when the mutated $c$-vector is $-e_k$. The lemma follows.
\end{proof}

\begin{thm}[Rotation Lemma]\label{rotation}
Let $B$ be a skew-symmetrizable matrix, let\break $(\kay_0,\kay_1,\dots,\kay_{m-1})$ be an $r$-reddening sequence for $B$ with associated permutation $\sigma$ and let $B'=\mu_{\kay_0}B$. Then 
the last $c$-matrix in the mutation sequence

\[
	(\kay_1,\kay_2,\dots,\kay_{m-1},\sigma^{-1}(\kay_0))
\]
is $-P_\sigma$, i.e., this is a reddening sequence for $B'$ with the same permutation as the reddening sequence for $B$. Furthermore, this new reddening sequence has exactly $r$ red mutations. In particular, a maximal green sequence for $B$ gives a maximal green sequence for $B'$.

\end{thm}

\begin{proof}
We use the same notation as in the mutation formula. The reddening sequence\break $(\kay_0,\dots,\kay_{m-1})$ gives the mutation sequences
\[
	\mat{B\\ I_n}=\mat{B_0\\ C_0}\xrightarrow{\mu_{\kay_0}}\mat{B_1\\C_1}\xrightarrow{\mu_{\kay_1}}\mat{B_2\\C_2}\xrightarrow{\phantom{\mu_{\kay_0}}}\cdots\xrightarrow{\phantom{\mu_{\kay_0}}}\mat{B_m\\C_m}
\]
By Nakanishi-Zelevinsky we have:
\[
	DB_s=C_s^tDB_0C_s
\]
Since $(\kay_0,\kay_1,\dots,\kay_{m-1})$ is a reddening sequence we have $C_m=-P$ where $P=P_\sigma$. By Lemma \ref{lem: CD=DC} we have that $P$ commutes with $D$. Thus, $B_m=D^{-1}P^tD B_0P=P^tB_0P$. Let $j=\sigma^{-1}(\kay_0)$ so that the $j$-th column of $P$ is the unit vector $e_{\kay_0}$. By the mutation formula in \ref{formula} we get:
\[
	\widetilde B'_m=\mat{B_m\\C_m'}=\mat{P^tB_0P\\-C_1P}.
\]

\noindent The $c$-matrix $C_1$ is equal to the identity matrix $I_n$ except for its $\kay_0$-th row which is given by
\[
	(C_1)_{\kay_0\ell}=\begin{cases} -1 & \text{if } \ell=\kay_0\\
    \max(b_{\kay_0\ell},0)& \text{if } \ell\neq\kay_0
    \end{cases}
\]

Multiplication by $-P_\sigma$ gives the matrix $-C_1P_\sigma$ which is equal to $-P_\sigma$ except for its $j$-th row ($j=\sigma^{-1}(\kay_0)$) where
\[
	-(C_1P_\sigma)_{j\ell}=\begin{cases} 1 & \text{if } \ell=j\\
   - \max(b_{\sigma(j)\sigma(\ell)},0)& \text{if } \ell\neq j
    \end{cases}
\]
But, $b_{\sigma(j)\sigma(\ell)}=b_{\kay_0\sigma(\ell)}$ is exactly the $(j,\ell)$ entry of the matrix $B_m=P_\sigma^tB_0P_\sigma$. When $b_{\sigma(j)\sigma(\ell)}>0$ we add $b_{\sigma(j)\sigma(\ell)}$ times the $j$-th column of $C_m'=-C_1P_\sigma$ (which is $e_{\kay_0}$) to its $\ell$-th column
\[
	-C_1P_\sigma e_\ell=-e_{\sigma(\ell)}-\max(b_{\kay_0\sigma(\ell)},0)e_{\kay_0}
\]
to get $-e_{\sigma(\ell)}$. Then we change the sign of the $j$-th column to produce the $c$-matrix $C_{m+1}'=-P_\sigma$. This proves that $(\kay_1,\dots,\kay_{m-1},\sigma^{-1}(\kay_0))$ is a reddening sequence for $B'$ with the same associated permutation $\sigma$.

It remains to show that this new reddening sequence has the same number of red mutations as the original reddening sequence.

Let $r\ge0$ be the number of red mutations in the first sequence. This includes $p$ mutations at the $c$-vector $-e_{\kay_0}$ and $q=r-p$ mutations at other negative $c$-vectors. By  lemma \ref{one-more-time}, there will be exactly $p+1$ mutations at the positive $c$-vector $e_{\kay_0}$. The first mutation will be one of these. Of the remaining $m-1$ mutations in the first mutation sequence, exactly $p$ will be at the $c$-vector $e_{\kay_0}$ and exactly $p$ will be at the $c$-vector $-e_{\kay_0}$.

By the mutation formula, the sign of the mutation $C_s\xrightarrow{\mu_{k_s}} C_{s+1}$ will be the same as the sign of the mutation $C_s'\to C_{s+1}'$ if the $c$-vector being mutated is not equal to $\pm e_{\kay_0}$. This means that both mutation sequences have the same number $q$ of mutations at negative $c$-vectors not equal to $-e_{\kay_0}$. The mutation formula also tells us that, if the $c$-vector being mutated in $C_s$ is $\pm e_{\kay_0}$ then the $c$-vector being mutated in $C_s'$ will be the negative of that vector. Thus, the $p$ red mutations at $-e_{\kay_0}$ for $C_s$ will become $p$ green mutations for $C_s'$ and vice versa. The number of red mutations for the rotated sequence will thus be $p+q=r$. (The last mutation is at the positive $c$-vector $e_{\kay_0}$.) This completes the proof of the Rotation Lemma.
\end{proof}

\begin{prop}
Let $(\kay_0,\dots,\kay_{m-1})$ be a reddening sequence for $B$ and let $t$ be maximal so that $\kay_t$ is a mutation on vector $e_{\kay_0}$. Let $c_0,\dots,c_{m-1}$ be the $c$-vector labeling of the same mutation sequence. (So, $c_0=c_t=e_{\kay_0}$.) Let $c_1',\dots,c_m'$ be the $c$-vector labeling of the rotated reddening sequence for $B'$. Then $c_\mutindex=C_1c_\mutindex'=X_{\kay_0}^+c_\mutindex'$ for all $t<\mutindex<m$.
\end{prop}

\begin{proof} Since a reddening sequence must end in $H_{\kay_0}^-$, it cannot leave the region after entering it for the last time. Therefore, for $\mutindex>t$, the $c$-matrix $C_\mutindex$ must remain in the negative part $H_{\kay_0}^-$ of the hyperplane $H_{\kay_0}$. The mutation formula then gives $c_\mutindex=C_1c_\mutindex'=X_{\kay_0}^+c_\mutindex'$ as claimed.
\end{proof}

\begin{figure}[h]  
	\begin{center}
	\newdimen\dd   
	\dd=3em		   
	
	\newdimen\sp   
	\sp=5\dd

\includegraphics{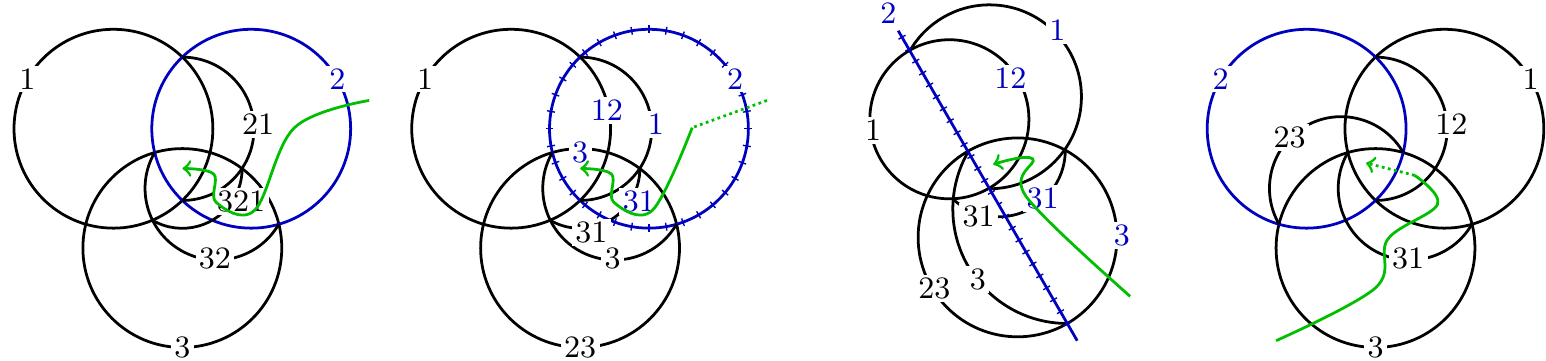}

    \end{center}
	
	\caption{Illustration of the Rotation Lemma. Rotating at $2$ turns the maximal green sequence $(2,3,1,3,2)$ of the quiver $Q:3\to 2\to 1$ into the maximal green sequence $(3,1,3,2,3)$ on $Q'=\mu_2Q$. In the second picture, the normal orientation of the sphere $D(2)$ is reversed and the wall labels change. In the third picture $D(2)$ expands out to ``infinity'' and in the fourth picture closes up with the correct normal orientation. The dotted green lines indicate the start of the old and the end of the new maximal green sequence.}
\label{fig:rotation}\label{fig03}
\end{figure}

\begin{eg}
We further illustrate the rotation lemma in the example from Figure \ref{fig:rotation}, namely rotating the maximal green sequence $(2,3,1,3,2)$ of $Q:3\to 2\to 1$ to the maximal green sequence $(3,1,3,2,3)$ of $Q'=\mu_2Q$. The corresponding $c$-matrix sequence for $Q$ is 

\begin{footnotesize}
\[
\left[\begin{array}{c>{\columncolor{green!25}}cc}
1 & 0 & 0 \\ 0 & 1 & 0 \\ 0 & 0 & 1 \\ 
\end{array}\right],
\left[\begin{array}{cc>{\columncolor{green!25}}c}
1 & 0 & 0 \\ 1 & -1 & 0 \\ 0 & 0 & 1 \\
\end{array}\right],
\left[\begin{array}{>{\columncolor{green!25}}ccc}
1 & 0 & 0 \\ 1 & -1 & 0 \\ 1 & 0 & -1 \\
\end{array}\right],
\left[\begin{array}{cc>{\columncolor{green!25}}c}
-1 & 0 & 1 \\ -1 & -1 & 1 \\ -1 & 0 & 0 \\
\end{array}\right],
\left[\begin{array}{c>{\columncolor{green!25}}cc}
0 & 1 & -1 \\ 0 & 0 & -1 \\ -1 & 0 & 0 \\
\end{array}\right],
\left[\begin{array}{ccc}
0 & -1 & 0 \\ 0 & 0 & -1 \\ -1 & 0 & 0 \\
\end{array}\right]
\]
\end{footnotesize}

where the highlighted columns indicate the $c$-vector where the mutation occurs.
Likewise the $c$-matrix sequence for $Q'$ is

\begin{footnotesize}
\[
\left[\begin{array}{cc>{\columncolor{green!25}}c}
1 & 0 & 0 \\ 0 & 1 & 0 \\ 0 & 0 & 1 \\
\end{array}\right],
\left[\begin{array}{>{\columncolor{green!25}}ccc}
1 & 0 & 0 \\ 0 & 1 & 0 \\ 1 & 0 & -1 \\
\end{array}\right],
\left[\begin{array}{cc>{\columncolor{green!25}}c}
-1 & 0 & 1 \\ 0 & 1 & 0 \\ -1 & 0 & 0 \\
\end{array}\right],
\left[\begin{array}{c>{\columncolor{green!25}}cc}
0 & 1 & -1 \\ 0 & 1 & 0 \\ -1 & 0 & 0 \\
\end{array}\right],
\left[\begin{array}{cc>{\columncolor{green!25}}c}
0 & -1 & 0 \\ 0 & -1 & 1 \\ -1 & 0 & 0 \\ 
\end{array}\right],
\left[\begin{array}{ccc}
0 & -1 & 0 \\ 0 & 0 & -1 \\ -1 & 0 & 0 \\
\end{array}\right].
\]
\end{footnotesize}

\end{eg}

\subsection{Greg Muller's example} 

This section provides some explanations why our Rotation Lemma does not contradict Greg Muller's example. 

At first sight, the Rotation Lemma might seem contradictory to Greg Muller's recent preprint \cite{M} where he provides examples showing that the existence of a maximal green sequence is not invariant under quiver mutation. However, as 0-reddening sequences are maximal green sequences, the Rotation Lemma shows: 
Given a maximal green sequence $(\kay_0,\kay_1,\dots,\kay_{m-1})$ on $Q$ with associated permutation $\sigma$,
\[
	(\kay_1,\kay_2,\dots,\kay_{m-1},\sigma^{-1}(\kay_0))
\]
is a maximal green sequence on $\mu_{\kay_0}Q$ with the same permutation $\sigma$. So the existence of maximal green sequences, and their respective length,  is in fact preserved under mutation {\it for those quivers that appear along the maximal green sequence.}
This yields a restriction on which quivers in the mutation class of $Q$ can appear along the maximal green sequence for $Q$, an effect that can already be illustrated for a quiver $Q$ of type $A_3$:
 
We reproduce Figure 5 from \cite{BDP} of the oriented mutation graph for the linear oriented quiver $Q$ of type $A_3$. Maximal green sequences are the oriented paths starting at the vertex labeled $\Lambda[1]$ and ending in the vertex labeled $\Lambda$. In particular, there are maximal green sequences of lengths 3, 4, 5 and 6.
	
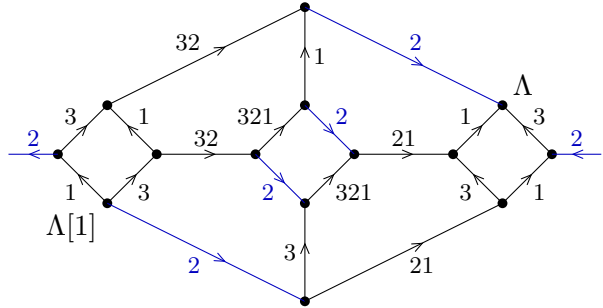
\begin{figure}[h]
\begin{center}

\begin{tikzpicture}[scale = .65]
	\begin{scope}[decoration={
		markings,
		mark=at position 0.6 with {\arrow{angle 45}}}
		] 
		\foreach \y in {0} 
		{
		\fill (3,\y+1) circle (.1);
		\fill (3,\y-1) circle (.1);

		\fill (7,\y+1) circle (.1);
		\fill (7,\y-1) circle (.1);

		\fill (11,\y+1) circle (.1);
		\fill (11,\y-1) circle (.1);

		\fill (7,\y-3) circle (.1);
		\fill (7,\y+3) circle (.1);

		\foreach \x in {2,4,6,8,10,12}
		{ 
		\fill (\x,\y) circle (.1);
		}

		\draw (3,\y-1) node[anchor=north east]{$\Lambda[1]$} ;
		\draw (11,\y+1) node[anchor=south west]{$\Lambda$} ;

		\draw[postaction={decorate}] (2,\y) -- (3,\y+1) node[annfar,midway,anchor=south east]{$3$} ; 
		\draw[postaction={decorate}] (4,\y) -- (3,\y+1) node[annfar,midway,anchor=south west]{$1$} ;

		\draw[postaction={decorate}] (3,\y-1) -- (2,\y) node[annfar,midway,anchor=north east]{$1$} ;
		\draw[postaction={decorate}] (3,\y-1) -- (4,\y) node[annfar,midway,anchor=north west]{$3$} ;

		\draw[postaction={decorate}] (4,\y) -- (6,\y) node[annfarer,midway,above]{$32$} ;

		\draw[postaction={decorate}] (6,\y) -- (7,\y+1) node[annfar,midway,anchor=south east]{$321$} ;
		\draw[postaction={decorate}, color=blue!75!black] (6,\y) -- (7,\y-1) node[annfar,midway,anchor=north east]{$2$} ;
		\draw[postaction={decorate}, color=blue!75!black] (7,\y+1) -- (8,\y) node[annfar,midway,anchor=south west]{$2$} ;
		\draw[postaction={decorate}] (7,\y-1) -- (8,\y) node[annfar,midway,anchor=north west]{$321$} ;

		\draw[postaction={decorate}] (8,\y) -- (10,\y) node[annfarer,midway,above]{$21$} ;

		\draw[postaction={decorate}] (10,\y) -- (11,\y+1) node[annfar,midway,anchor=south east]{$1$} ;
		\draw[postaction={decorate}] (11,\y-1) -- (10,\y) node[annfar,midway,anchor=north east]{$3$} ;
		\draw[postaction={decorate}] (12,\y) -- (11,\y+1) node[annfar,midway,anchor=south west]{$3$} ;
		\draw[postaction={decorate}] (11,\y-1) -- (12,\y) node[annfar,midway,anchor=north west]{$1$} ;

		\draw[postaction={decorate}, color=blue!75!black] (2,\y) -- (1,\y) node[annfarer,midway,above]{$2$} ;
		\draw[postaction={decorate}, color=blue!75!black] (13,\y) -- (12,\y) node[annfarer,midway,above]{$2$} ;

		\draw[postaction={decorate}] (3,\y+1) -- (7,\y+3) node[annfar,midway,anchor=south east]{$32$} ;
		\draw[postaction={decorate}] (7,\y+1) -- (7,\y+3) node[annfarer,midway,right]{$1$} ;
		\draw[postaction={decorate}, color=blue!75!black] (7,\y+3) -- (11,\y+1) node[annfar,midway,anchor=south west]{$2$} ;

		\draw[postaction={decorate}, color=blue!75!black] (3,\y-1) -- (7,\y-3) node[annfar,midway,anchor=north east]{$2$} ;
		\draw[postaction={decorate}] (7,\y-3) -- (7,\y-1) node[annfarer,midway,left]{$3$} ;
		\draw[postaction={decorate}] (7,\y-3) -- (11,\y-1) node[annfar,midway,anchor=north west]{$21$} ;

		}
	\end{scope}
					
\end{tikzpicture}			
\caption{The oriented mutation graph of $Q:3\to 2\to 1$.}\label{fig04}
\label{oriented-graph3}
\end{center}
\end{figure}

On the other hand, the cyclic oriented simple graph $Q'=\mu_2Q$ with three vertices yields the oriented mutation graph shown in Figure 6. We see that the maximal green sequences are of lengths 4 or 5 in this case, even if the quivers $Q$ and $Q'$ are mutation equivalent. So it is certainly not the case that {\it every} maximal green sequence of $Q$ yields a maximal green sequence (of the same length) of the mutated quiver $Q'$.
The structural result implied by the Rotation Lemma is: The cyclically oriented quiver $Q'$ cannot occur along a maximal green sequence for $Q$ of length 3 or 6. In fact, the only quivers appearing along these maximal green sequences are the acyclic quivers of type $A_3$, and those admit sequences of length 3 and 6. Thus the membership to certain maximal green sequences yields a finer subdivision of the mutation class of $Q$.

Note how the Rotation Lemma can be observed in the figures: the orientation of the overall graph changes by applying mutation, but along a maximal oriented path of length $n$ the orientation stays the same in the last $n-1$ arrows.

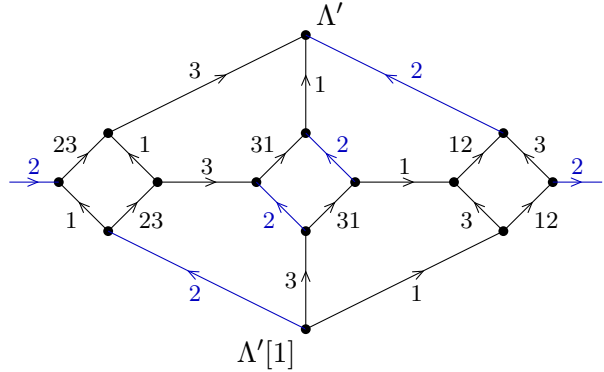
\begin{figure}[h]
\begin{center}

\begin{tikzpicture}[scale = .65]
	\begin{scope}[decoration={
		markings,
		mark=at position 0.6 with {\arrow{angle 45}}}
		] 
		\foreach \y in {0} 
		{
		\fill (3,\y+1) circle (.1);
		\fill (3,\y-1) circle (.1);

		\fill (7,\y+1) circle (.1);
		\fill (7,\y-1) circle (.1);

		\fill (11,\y+1) circle (.1);
		\fill (11,\y-1) circle (.1);

		\fill (7,\y-3) circle (.1);
		\fill (7,\y+3) circle (.1);

		\foreach \x in {2,4,6,8,10,12}
		{ 
		\fill (\x,\y) circle (.1);
		}

		\draw (7,\y-3) node[anchor=north east]{$\Lambda'[1]$} ;
		\draw (7,\y+3) node[anchor=south west]{$\Lambda'$} ;

		\draw[postaction={decorate}] (2,\y) -- (3,\y+1) node[annfar,midway,anchor=south east]{$23$} ; 
		\draw[postaction={decorate}] (4,\y) -- (3,\y+1) node[annfar,midway,anchor=south west]{$1$} ;

		\draw[postaction={decorate}] (3,\y-1) -- (2,\y) node[annfar,midway,anchor=north east]{$1$} ;
		\draw[postaction={decorate}] (3,\y-1) -- (4,\y) node[annfar,midway,anchor=north west]{$23$} ;

		\draw[postaction={decorate}] (4,\y) -- (6,\y) node[annfarer,midway,above]{$3$} ;

		\draw[postaction={decorate}] (6,\y) -- (7,\y+1) node[annfar,midway,anchor=south east]{$31$} ;
		\draw[postaction={decorate}, color=blue!75!black] (7,\y-1) -- (6,\y) node[annfar,midway,anchor=north east]{$2$} ;
		\draw[postaction={decorate}, color=blue!75!black] (8,\y) -- (7,\y+1) node[annfar,midway,anchor=south west]{$2$} ;
		\draw[postaction={decorate}] (7,\y-1) -- (8,\y) node[annfar,midway,anchor=north west]{$31$} ;

		\draw[postaction={decorate}] (8,\y) -- (10,\y) node[annfarer,midway,above]{$1$} ;

		\draw[postaction={decorate}] (10,\y) -- (11,\y+1) node[annfar,midway,anchor=south east]{$12$} ;
		\draw[postaction={decorate}] (11,\y-1) -- (10,\y) node[annfar,midway,anchor=north east]{$3$} ;
		\draw[postaction={decorate}] (12,\y) -- (11,\y+1) node[annfar,midway,anchor=south west]{$3$} ;
		\draw[postaction={decorate}] (11,\y-1) -- (12,\y) node[annfar,midway,anchor=north west]{$12$} ;

		\draw[postaction={decorate}, color=blue!75!black] (1,\y) -- (2,\y) node[annfarer,midway,above]{$2$} ;
		\draw[postaction={decorate}, color=blue!75!black] (12,\y) -- (13,\y) node[annfarer,midway,above]{$2$} ;

		\draw[postaction={decorate}] (3,\y+1) -- (7,\y+3) node[annfar,midway,anchor=south east]{$3$} ;
		\draw[postaction={decorate}] (7,\y+1) -- (7,\y+3) node[annfarer,midway,right]{$1$} ;
		\draw[postaction={decorate}, color=blue!75!black] (11,\y+1) -- (7,\y+3) node[annfar,midway,anchor=south west]{$2$} ;

		\draw[postaction={decorate}, color=blue!75!black] (7,\y-3) -- (3,\y-1) node[annfar,midway,anchor=north east]{$2$} ;
		\draw[postaction={decorate}] (7,\y-3) -- (7,\y-1) node[annfarer,midway,left]{$3$} ;
		\draw[postaction={decorate}] (7,\y-3) -- (11,\y-1) node[annfar,midway,anchor=north west]{$1$} ;

		}
	\end{scope}
					
\end{tikzpicture}

\caption{The oriented mutation graph of the cyclic quiver $Q'$ with 3 vertices.}\label{fig05}
\label{oriented-graph4}
\end{center}
\end{figure}

The situation becomes more dramatic in Greg Muller's example where we consider the acyclic quiver $Q_{0,1,2}:3 \rightrightarrows 2 \to 1$. It does admit the following maximal green sequences: $(2,1,3,2)$ which we already considered in the introduction, as well as $(2,1,2,3)$ and the minimal sink reflection sequence $(1,2,3)$. 
The quivers appearing along these maximal green sequences are $Q_{0,1,2}$ and its source-sink reflections, as well as $ Q\op_{2,1,2}$, where we denote by $Q_{a,b,c}$ the cyclic quiver with $a$ arrows from 1 to 3, $b$ arrows from 2 to 1 and $c$ arrows from 3 to 2.
These are exactly the quivers in the mutation class of $Q_{2,1,0}$ that admit maximal green sequences. The mutation class of $Q_{2,1,0}$ is infinite, and the minimal quiver in this class not having a maximal green sequence is the quiver $Q_{2,3,2}$ discussed in Corollary 2.3.3 of \cite{M}.
However, as illustrated in Figures 11 and 18 of \cite{M}, the quiver $Q_{2,3,2}$ has a 1-reddening sequence of length 6 passing through $Q_{2,1,0}$.


\section{The Target before Source Conjecture}

For proving the general form of the Target before Source Conjecture, it will be convenient to introduce the following notion.

\begin{defn}
The \emph{maximal green tail} of a reddening sequence $(\kay_0,\kay_1,\dots,\kay_{m-1})$ is the subsequence $(\kay_\ell,\kay_{\ell+1}, \dots,\kay_{m-1})$ where $\kay_{\ell-1}$ is the last red mutation.
\end{defn}

Note that the maximal green tail of a reddening sequence need not itself be a maximal green sequence, as illustrated in the following example.

\begin{eg}
In the Kronecker quiver $2\rightrightarrows1$ the sequence $(1,2,1,1)$ is a reddening sequence. Its maximal green tail is the single mutation $(1)$ which is green, but not maximal.
\end{eg}

\begin{defn}
For an arrow $\alpha:\jay\to\aye$ of a valued quiver $Q$, denote by $Q[\alpha]$ the rank 2 quiver consisting of the single arrow $\alpha:\jay\to\aye$ with the same valuation as in $Q$.\\
An arrow $\alpha:\jay\to\aye$ with valuation $(d_{\jay\aye},d_{\aye\jay})$ is of \emph{infinite type} if $d_{\jay\aye}d_{\aye\jay}\ge 4$, or equivalently if $Q[\alpha]$ is representation infinite.
\end{defn}

\subsection{Recursion for rank 2 preinjective roots}

In order to prove the Target before Source Conjecture we need to relate the $c$-vectors of the reddening sequence to the preinjective roots of the rank 2 quiver $Q[\alpha]$ where $\alpha$ is an infinite type arrow of $Q$. Throughout this subsection we fix an infinite type arrow $\alpha:\jay\to\aye$ of $Q$. For simplicity, set $a=d_{\aye\jay}$ and $b=d_{\jay\aye}$.

The preinjective roots of $Q[\alpha]$ are linearly ordered by position in the Auslander-Reiten quiver. Denote by $q_t$ the root of $Q$ obtained by extending the $t$-th preinjective root of $Q[\alpha]$ by zero on vertices $k\ne \aye,\jay$. In particular we have $q_0=e_{\jay}$. 

In order to give a uniform description of the vectors $q_t$, we introduce the following family of polynomials.

\begin{defn}\label{chebyshev}
Define a family of Chebyshev-like polynomials $U_n(x,y)$ for $n\ge -1$ by $U_{-1}(x,y)=0$, $U_0(x,y)=1$ and for $n\ge 1$ by the recursion
\[
U_n(x,y)=xU_{n-1}(y,x)-U_{n-2}(x,y).
\]
\end{defn}

\begin{rem}
The ordinary Chebyshev polynomials (of the second kind) $U_n(x)$ are recovered from the $U_n(x,y)$ by the substitution $x,y\mapsto x/2$. The normalization factor ${1\over2}$ is chosen to simplify the following formula for the roots $q_t$.
\end{rem}

\begin{lem}
The roots $q_t$ of $Q$ have coordinates $q_t(\aye)=U_{t-1}(b,a)$, $q_t(\jay)=U_t(a,b)$ and $q_t(k)=0$ for $k\neq \jay,\aye$.
\end{lem}

\begin{proof}
Number the vertices of $Q$ so that $\aye=1$ and $\jay=2$. With this numbering, the Auslander-Reiten translate $\tau$ of $Q$ is given on dimension vectors  by
\[
\tau=\mat{-1 & b & * \\ -a & ab-1 & * \\ * & * & * }.
\]
In the quiver $Q[\alpha]$ the injective roots are $\undim I_2=\mat{ 0 \\ 1}$ and $\undim I_1=\mat{ 1 \\ a}$, so the lemma is true for $t=0,1$. For $t\ge2$ the roots $q_t$ are related by $q_t=\tau q_{t-2}$. Thus the coordinates of the $q_t$ satisfy the simultaneous recursion
\begin{equation*}
\begin{aligned}
q_t(\aye)&=bq_{t-2}(\jay)-q_{t-2}(\aye) \\
q_t(\jay)&=(ab-1)q_{t-2}(\jay)-aq_{t-2}(\aye)=aq_t(\jay)-q_{t-2}(\jay).
\end{aligned}
\end{equation*}

By induction we have 
\begin{equation*}
\begin{aligned}
q_t(\aye)&=bU_{t-2}(a,b)-U_{t-3}(b,a)=U_{t-1}(b,a) \\
q_t(\jay)&=aU_{t-1}(b,a)-U_{t-2}(a,b)=U_t(a,b)
\end{aligned}
\end{equation*}
proving the lemma.
\end{proof}

Denote by $q'_t$ for $t\ge 0$ the extension of the preinjective roots of $Q'[\alpha']$ to $Q'=\mu_{\jay}Q$ where $\alpha':\aye\to\jay$. Denote by $q_{-1}$ the vector with $q_{-1}(\aye)=-1$ and all other coordinates 0. The following lemma explains the relationship between the roots $q_t$ of $Q$ to the roots $q'_t$ of $Q'$.

\begin{lem}\label{preinjectives}
For every $t\ge0$ one has $q'_{t-1}=X_{\jay}^+q_t$ where $X_{\jay}^+$ is as in Definition \ref{Xmatrix}.
\end{lem}

\begin{proof}
Upon numbering the vertices of $Q$ so that $1=\aye$ and $2=\jay$, the matrix $X_\jay^+$ is given by 
\[
X_{\jay}^+=\mat{1 & 0 & 0 \\ a & -1 & * \\ 0 & 0 & I_{n-2}}
\]
and so
\[
X_{\jay}^+q_t=\mat{1 & 0 & 0 \\ a & -1 & * \\ 0 & 0 & I_{n-2}}\mat{U_{t-1}(b,a) \\ U_t(a,b) \\ 0}=\mat{U_{t-1}(b,a) \\ aU_{t-1}(b,a)-U_t(a,b) \\ 0}=\mat{U_{t-1}(b,a) \\ U_{t-2}(a,b) \\ 0}=q'_{t-1}
\]
provided that $t>0$. When $t=0$ one calculates
\[
X_{\jay}^+q_t=\mat{1 & 0 & 0 \\ a & -1 & * \\ 0 & 0 & I_{n-2}}\mat{0 \\ 1 \\ 0}=\mat{0 \\ -1 \\ 0}=q'_{-1}
\]
proving the lemma.
\end{proof}

\subsection{Target before source for reddening sequences}

To prove the Target before Source Conjecture we first prove the more general Theorem \ref{thm:ibeforej} which states that mutation at the target comes before mutation at the source in the green tail of a reddening sequence.

\begin{lem}\label{lemmaD}
Suppose $Q$ is a valued quiver, $\alpha:\jay\to\aye$ is a fixed arrow and $\mgs{\kay}$ is a reddening sequence. Consider the two reddening sequences
\begin{enumerate}
\renewcommand{\theenumi}{\alph{enumi}}
\item $\mgs{\kay}'=(\jay,\jay,\kay_0,\kay_1,\dots,\kay_{m-1})$ which is a reddening sequence of $Q$
\item $\mgs{\kay}''$ which is the reddening sequence of $Q'=\mu_{\jay}Q$ given by rotating $\mgs{\kay}'$.
\end{enumerate}

\noindent If the $c$-vectorx $e_{\jay},e_\aye$ occurs at vertices $k_\ell,k_p$ in the maximal green tail of $\mgs{\kay}'$ with $p>\ell$, then:
\begin{enumerate}
\item $C_\ell\in H_\aye^+\cap H_\jay^+$
\item $C'_\ell\in H_\aye^+\cap H_\jay^-$
\item the $c$-vector $e_{\aye}$ occurs before the $c$-vector $e_{\jay}$ in the maximal green tail of $\mgs{\kay}''$ (\emph{i.e.}, in the opposite order in which they occur in $\mgs{\kay}'$).
\end{enumerate}
\end{lem}

\begin{proof}

(1) By Lemma \ref{hemispheres}, the $c$-matrix $C_\ell\in H_\aye^+$ and $C_p\in H_\jay^+$. However, all of the mutations in the maximal green tail of $\mgs{\kay}$ are green and so the $c$-matrices cannot go from $H_{\jay}^-$ to $H_{\jay}^+$. Since $\ell<p$, we have $C_\ell\in H_{\jay}^+$ as claimed.

(2) Since $C_\ell\in H_\jay^+$, we have $C_\ell'=X_\jay^-C_\ell$ where
\[ 
X_\jay^+=\mat{1 & 0 & 0 \\ a & -1 & * \\ 0 & 0 & I_{n-2}} \qquad \text{ and } \qquad X_\jay^-=\mat{1 & 0 & 0 \\ 0 & -1 & * \\ 0 & 0 & I_{n-2}}
\]
(assuming $\aye=1$, $\jay=2$ for simplicity). So, $(G'_\ell)^t=D(C'_\ell)^{-1}D^{-1}=DC_\ell^{-1}X_\jay^{-1}D^{-1}=G_\ell^tDX_j^-D^{-1}$. By (1), columns $\aye$ and $\jay$ of $G_\ell^t$ are positive. So, columns $\aye$ and $\jay$ of $(G'_\ell)^t$ are positive and negative, respectively. This is equivalent to (2).

(3) By (2) $C_\ell'\in H^-_\jay$. Since all mutations in the maximal green tail of $\mgs{\kay}$ are green, the last $c$-matrix is in $H_i^-$. By Lemma \ref{hemispheres} there is a $q>\ell$ so that the mutation $\kay_q$ of $C_q'$ is at $c$-vector $e_i$. The last mutation of $\mgs{\kay}''$ is at the $c$-vector $e_j$ by rotation, proving the lemma.
\end{proof}

We now turn to the main theorem of this section.

\begin{thm}\label{thm:ibeforej}
Suppose that $Q$ is a valued quiver having an infinite type arrow $\alpha:\jay\to\aye$ and $\mgs{\kay}=(\kay_0,\kay_1,\dots,\kay_{m-1})$ is a reddening sequence in which the $c$-vector sequence of the maximal green tail of $\mgs{\kay}$ contains the simple roots $e_i$ and $e_j$. Then $e_{\aye}$ must occur before $e_{\jay}$.
\end{thm}

\begin{proof}
Suppose $\mgs{\kay}$ is a reddening sequence for $Q$ in which $e_{\jay}$ occurs before $e_{\aye}$ in the maximal green tail of $\mgs{\kay}$. Let $\kay_\ell$ be the first vertex in the maximal green tail of $\mgs{\kay}$ with corresponding $c$-vector $e_{\jay}$.

We claim that for each integer $s\ge 0$ there is a subsequence 
\[
(\kay_{\ell_0},\kay_{\ell_1},\dots,\kay_{\ell_s})
\]
of the maximal green tail of $\mgs{\kay}$ with $\kay_{\ell_0}=\kay_\ell$ and corresponding $c$-vectors $c_{\ell_t}=q_t$ for $0\le t\le s$. This provides a contradiction, as the sequence $\mgs{\kay}$ is finite.

The claim is proven by induction on $s$. The statement holds when $s=0$ since $\kay_\ell=\kay_{\ell_0}=e_{\jay}=q_0$ by definition. Suppose that the claim holds for some $s$. Consider the extended reddening sequence $\mgs{\kay}'=(\jay,\jay,\kay_0,\kay_1,\dots,\kay_{m-1})$ as in Lemma \ref{lemmaD}.

By induction, there is a subsequence $(\kay_{\ell_0},\dots,\kay_{\ell_s})$ of the maximal green tail of $\mgs{\kay}$ (which is the same as that of $\mgs{\kay}'$), with corresponding $c$-vectors $q_t$ for $0\le t\le s$. By Lemma \ref{preinjectives} the subsequence $(\kay_1,\kay_2,\dots,\kay_s)$ of the tail of the rotated sequence $\mgs{\kay}''$ has associated $c$-vectors $q'_0,q'_1,\dots, q'_{s-1}$. Moreover, by Lemma \ref{lemmaD} $c'_{\kay_1}=e_{\aye}$ and mutation at $e_{\jay}$ occurs after $\kay_1$ so by induction there is a vertex $\kay_{\ell_{s+1}}$ of the maximal green tail of $\mgs{\kay}''$ with corresponding $c$-vector $q'_s$. Since the matrix $X_{\jay}^+$ satisfies $X_{\jay}^+X_{\jay}^+=I$, the $c$-vector of $Q$ corresponding to $\kay_{\ell_{s+1}}$ of the unrotated sequence $\mgs{\kay}'$ is $X_{\jay}^+q'_s=
q_{s+1}$. Thus by induction, the claim holds.
\end{proof}

\subsection{Proof of Target before Source Conjecture}

This conjecture derives its name from Corollary \ref{cor:ibeforej}, which follows from Theorem \ref{thm:ibeforej} using two lemmas.

\begin{lem}\label{lemmaA}
If $Q$ is any valued quiver then any maximal green sequence mutates at each simple root $e_k$ exactly once.
\end{lem}

\begin{proof} 
A maximal green sequence crosses each hyperplane $H_k$. By Lemma \ref{hemispheres}, crossing $H_k$ amounts to mutating at the simple root $\pm e_k$. 
Since only green mutations are being performed, all of these mutations must be at $+e_k$. The maximal green sequence starts on the $+$-side of all hyperplanes, end on the $-$-side of all hyperplanes, and so must mutate each $e_k$.
\end{proof}

\begin{cor}\label{cor:general conjecture 1}
Consider any maximal green sequence on any valued quiver $Q$. Then, at each step, the mutation is at a vertex of the mutated quiver $Q'$ which is not the source of any arrow of infinite type.
\end{cor}

\begin{proof}
If this occurs, use the Rotation Lemma to make $Q'$ the initial quiver. Then we have an arrow of infinite type $j\to i$ and the first mutation is at $c$-vector $e_j$. There must be a mutation at $e_i$ later by the lemma above, contradicting Theorem \ref{thm:ibeforej}.
\end{proof}

We now restrict to acyclic quivers since, by the example in Figure \ref{fig02}, the following lemma and corollary do not hold for quivers with oriented cycles.

\begin{lem}\label{lemmaB}
Suppose $Q$ is an acyclic valued quiver having an arrow $\alpha:\jay\to\aye$ of infinite type. If $\kay_1,\kay_2,\dots, \kay_s$ is any sequence of vertices with each $\kay_t\ne\aye,\jay$, then $Q'=\mu_{\kay_s}\mu_{\kay_{s-1}}\cdots\mu_{\kay_1}Q$ has an arrow $\alpha':\jay\to\aye$ of infinite type.
\end{lem}

\begin{proof}
Let $T=P_1\oplus\cdots\oplus P_n$ be the projective cluster in the cluster category $\cC_Q$ of $Q$, and let $T'$ be the cluster-tilting object given by $\mu_{\kay_s}\mu_{\kay_{s-1}}\cdots\mu_{\kay_1}T$ (\emph{cf.}, \cite{BMRRT}). The quiver $Q'$ is the Gabriel quiver of the cluster tilted algebra $B=\End_{\cC_Q}(T')\op$, and the valuation $(d'_{\jay\aye},d'_{\aye\jay})$ of $\alpha':\jay\to\aye$ in $Q'$ is given by $d'_{\jay\aye}=\dim_{F_{\aye}}\Irr_{B}(P_{\aye},P_{\jay})$ and $d'_{\aye\jay}=\dim_{F_{\jay}}\Irr_{B}(P_{\aye},P_{\jay})$ where $\Irr_B(P_{\aye},P_{\jay})$ denotes the space of irreducible $B$-linear maps from $P_{\aye}\to P_{\jay}$.

Since $\Irr_{B}(P_{\aye},P_{\jay})$ is the quotient of $\Irr_{\cC_Q}(P_{\aye},P_{\jay})$ by the ideal of morphisms $P_{\aye}\to P_{\jay}$ factoring through objects in the cluster $T'$ not equal to $P_{\jay}$ or $P_{\aye}$, the natural map $\Irr_{Q}(P_{\aye},P_{\jay})\to\Irr_{B}(P_{\aye},P_{\jay})$ is surjective. Hence $d'_{\jay\aye}d'_{\aye\jay}\ge d_{\jay\aye}d_{\aye\jay}$. In particular, if $\alpha:\jay\to\aye$ is infinite type in $Q$, $\alpha':\aye\to\jay$ is infinite type in $Q'$.
\end{proof}

\begin{cor}[Target before Source Conjecture]\label{cor:ibeforej}
If $Q$ is an acyclic valued quiver with an infinite type arrow $\alpha:\jay\to\aye$, any maximal green sequence mutates at the vertex $\aye$ before the vertex $\jay$.
\end{cor}

\begin{proof}
Suppose that the first occurrence of $\jay$ precedes the first occurrence of $\aye$ in the maximal green sequence of $Q$. Rotate the sequence to form a maximal green sequence of a quiver $Q'$ having $\jay$ as the first mutation. By Lemma \ref{lemmaB} the quiver $Q'$ still has an infinite type arrow $\alpha':\jay\to\aye$.
The first mutation of the rotated sequence occurs at the $c$-vector $e_{\jay}$. By Lemma \ref{lemmaA}, the rotated sequence eventually mutates at the $c$-vector $e_{\aye}$. Since $Q'$ has an infinite type arrow $\alpha':\jay\to\aye$, this contradicts Theorem \ref{thm:ibeforej}, proving the corollary.
\end{proof}


\section{Finite number of reddening sequences}

In this section we prove the following theorem.

\begin{thm}\label{thm: tame quivers have finitely many r-red sequences}
If $Q$ is a quiver which is mutation equivalent to an acyclic tame quiver then $Q$ has at most finitely many $r$-reddening sequences for every $r\ge0$. In particular, $Q$ has at most finitely many maximal green sequences.
\end{thm}

By the Rotation Lemma, it suffices to prove the theorem in the case when $Q$ is any acyclic tame (valued) quiver. The proof uses domains of semi-invariants and the easy observation that every cluster contains at least one preprojective or preinjective component. We begin with the basic definitions and an outline of the proof.

\subsection{Definitions and outline of proof}

Let $\Lambda$ be a fixed tame hereditary algebra which is finite dimensional over a field $K$. Let $P_1,\cdots,P_n$ be the indecomposable projective $\Lambda$-modules. The dimension vectors of indecomposable modules are called the \emph{positive roots} of $\Lambda$. Let $\pi_i=\undim P_i$ be the \emph{projective roots}. We also consider negative roots such as $-\pi_i=\undim P_i[1]$. The (positive) \emph{real Schur roots} of $\Lambda$ are the dimension vectors of indecomposable rigid modules (also called \emph{exceptional modules}). Since rigid modules are determined by their dimension vectors we write $M_\beta$ for the exceptional module with dimension vector $\beta$.

We recall that Auslander-Reiten translation $\tau$ is given on nonprojective roots by:
\[
	\tau \beta=-E^{-1}E^t\beta
\]
where $E$ is the Euler matrix of $\Lambda$. Note that $-\tau\pi_i$ is the dimension vector of the $i$-th injective module. The matrix $-E^{-1}E^t$ is an invertible integer matrix. Recall that the \emph{Euler-Ringel pairing} $\brk{\cdot,\cdot}:\ZZ^n\times \ZZ^n\to \ZZ$ is given by $\brk{x,y}=x^tEy$ and, for all modules $M,N$, we have
\[
	\brk{\undim M,\undim N}=\dim_K\Hom_\Lambda(M,N)-\dim_K\Ext^1_\Lambda(M,N).
\]
We have Auslander-Reiten duality: 
$
	\left<\alpha,\tau\beta\right>=-\left<\beta,\alpha\right>
$ and $\tau$ is an isometry:

$
	\left<\tau\alpha,\tau\beta\right>=\left<\alpha,\beta\right>.
$

Let $\cP_1$ be the set of projective roots and, for all $k\ge1$, let $\cP_k$ be the set of preprojective roots given by $\cP_k:=\cP_1\cup \tau^{-1}\cP_1\cup \cdots\cup \tau^{-(k-1)}\cP_1$.

Similarly, let $\cI_k:=\cI_1\cup \tau\cI_1\cup\cdots\cup \tau^{k-1}\cI_1$ where $\cI_1$ is the set of injective roots. Notice that $\cP_k,\cI_k$ are both finite with $kn$ elements.

For every $k\ge1$ let $\cW_k$ be subsets of $\RR^n$ defined by
\[
	\cW_k:=\{x\in \RR^n\,:\, \brk{x,\alpha}> 0\text{ for some }\alpha\in\cP_k\}.
\]
It is easy to see that the only roots in $\cW_k$ are those in $\cP_k$.
Note that the complement of $\cW_k$ is
\[
	\RR^n\backslash \cW_k=\{x\in \RR^n\,:\, \brk{x,\alpha}\le 0\text{ for all }\alpha\in\cP_k\}.
\]

For every $k\ge1$ let $\cV_k\subseteq\RR^n$ be defined by
\[
	\cV_k:=\{x\in \RR^n\,:\, \brk{x,\beta}\ge 0\text{ for all }\beta\in\cI_k\}.
\]
This set contains all roots except for those in $\cI_{k-1}$ since $\brk{x,\beta}<0$ for some $\beta\in\cI_k$ is equivalent to the statement $\brk{\gamma,x}>0$ for some $\gamma\in\cI_{k-1}$ (letting $\gamma=\tau^{-1}\beta$).

For each cluster tilting object $T=\bigoplus_{i=1}^n T_i$ in the cluster category of $\Lambda$ we have the simplicial cone 
\[
	R(T):=\left\{\sum_{i=1}^n a_i\undim T_i\,:\, a_i\ge0\right\} \subseteq\RR^n.
\]
We use properties of this set proved in \cite{IOTW2}.
Recall that, for distinct $T,T'$, the interiors of the regions $R(T),R(T')$ do not intersect. This follows from \cite{IOTW2}, Theorem 4.1.5. 

Let $Q$ be the valued quiver of $\Lambda$ given by the Euler matrix $E$. Let $B$ be the corresponding exchange matrix. We use the $c$-vector theorem from \cite{IOTW2} which implies the following.

\begin{thm}\label{c-vector theorem} Given any reddening sequence $(k_0,k_1,\cdots,k_{m-1})$ with corresponding exchange matrices $B=B_0,B_1,\cdots,B_m$ and $c$-matrices $I_n=C_0,C_1,\cdots,C_m=-P_\sigma$, there are unique cluster tilting objects $T^j=\bigoplus_{i=1}^n T_i^j$ for each $0\le j\le m$ so that 
\[
	V_j^tEC_j=-D
\]
where $V_j$ is the $n\times n$ matrix whose $i$-th column is $\undim T^j_i$ and $D$ is the diagonal matrix with diagonal entries $f_j=\dim_K\End_\Lambda(P_j)$, the valuation of $Q$ at vertex $j$. Also,
\begin{enumerate}
\item $T^0=\Lambda[1]$ and $T^m=\Lambda$. I.e., the first cluster tilting object in the reddening sequence is $\Lambda[1]$ and the last one is $\Lambda$.
\item $T^{j+1}=\mu_{k_j}T^j$ for every $j$. So, $T^j$ and $T^{j+1}$ differ only in their $k_j$-th component.
\item $R(T^j)\cap R(T^{j+1})$ is the subset of $H(\beta_k)=\{x\in\RR^n\,:\, \brk{x,\beta_k}=0\}$ spanned by $\undim T^j_i$ for $i\neq k_j$ where $\beta_k$ is the unique positive real Schur root so that the $k_j$-th column of $C_j$ is $\pm\beta_k$.
\item The mutation $\mu_{k_j}: B_j\mapsto B_{j+1}$ is green if and only if $R(T^j)$ is on the negative side of the hyperplane $H(\beta_k)$, i.e., $\brk{x,\beta_k}\le0$ for all $x\in R(T^j)$ and $\brk{y,\beta_k}\ge0$ for all $y\in R(T^{j+1})$.
\end{enumerate}
\end{thm}

\noindent In general, $R(\mu_kT)\cap R(T)$ is a subset of the ``semi-invariant domain'' (from \cite{IOTW2})
\[
	D(\beta_k)=\{x\in\RR^n\,:\, \brk{x,\beta_k}=0, \brk{x,\beta'}\le0\text{ for all subroots } \beta'\subset\beta_k\}
\]
of the real Schur root $\beta_k$ where by a \emph{subroot} of $\beta_k$ we mean the dimension vector of an indecomposable submodule of the exceptional module $M_{\beta_k}$ which is characterized by the property that $\Hom_\Lambda(T_i,M_{\beta_k})=0=\Ext_\Lambda^1(T_i,M_{\beta_k})$ for all $i\neq k$. Furthermore, the interiors of the regions $R(T)$ are disjoint from all $D(\beta)$. By the Virtual Stability Theorem \cite{IOTW2}, the condition $\brk{x,\beta'}\le0$ for all subroots $\beta'\subset\beta$ is equivalent to the condition that $\brk{x,\beta'}\le0$ for all real Schur subroots $\beta'\subseteq \beta$. This is also clearly equivalent to the condition that $\brk{x,\beta''}\ge0$ for all quotient roots $\beta''$ of $\beta$ and this is used in the next proof.

\begin{prop}\label{prop: R(T) does not cross dVk, dWk}
For every $k>0$ and every cluster tilting object $T$, the interior of $R(T)$ is either contained in $\cV_k$ or is disjoint from $\cV_k$. Similarly, the interior of $R(T)$ is either contained in $\cW_k$ or is disjoint from $\cW_k$ for every $k\ge1$.
\end{prop}

\begin{proof}
It suffices to show that the boundary of $\cV_k$ (its closure minus its interior) is a union of $D(\beta)$'s. So, let $x\in\partial\cV_k$. Then $\brk{x,\beta}=0$ for some $\beta\in\cI_k$. By definition of $\cV_k$, we have $\brk{x,\gamma}\ge0$ for all $\gamma\in\cI_k$. But this includes all quotient roots of $\beta$. Therefore, $x\in D(\beta)$ proving the claim. By an analogous argument applied to( $\RR^n\backslash \cW_k$ we see that $\partial\cW_k=\partial(\RR^n\backslash \cW_k)$ is also contained in a union of $D(\beta)$'s. The proposition follows.
\end{proof}

Since $\cV_k$ and $\cW_k$ lie on the positive side of $D(\beta)$ at each point on their boundaries and any mutation from the positive to the negative side of $D(\beta)$ is a red mutation by the $c$-vector threorem \ref{c-vector theorem}, we get the following.

\begin{cor}
Any mutation from a cluster tilting object $T$ inside $\cV_k$ or $\cW_k$ (i.e., so that the interior of $R(T)$ is inside the region) to one outside the region is red.\qed
\end{cor}

Since $\cV_k$ and $\cW_k$ contain all projective roots $\pi_i$ and none of the negative projective roots $-\pi_i$, it follows from Proposition \ref{prop: R(T) does not cross dVk, dWk} that the interior of $R(\Lambda)$ lies in $ \cV_k\cap \cW_k$ and the interior of $R(\Lambda[1])$ lies outside $\cV_k\cup\cW_k$ for all $k\ge1$. Thus every reddening sequence begins outside of both $\cV_k$ and $\cW_k$ and ends inside of both for all $k\ge1$. It is important to know which region the reddening sequence enters first, $\cV_k$ or $\cW_k$.

\begin{defn}
We say that a reddening sequence for $\Lambda$ \emph{meets} $\cV_k\backslash\cW_k$ if there is a cluster $T$ in the sequence so that $R(T)\subseteq\cV_k\backslash\cW_k$. If this is not the case, Proposition \ref{prop: R(T) does not cross dVk, dWk} implies that the interior of each $R(T)$ in the mutation sequence is disjoint from $\cV_k\backslash\cW_k$ and we say that the reddening sequence is \emph{disjoint} from $\cV_k\backslash\cW_k$.
\end{defn}

\begin{rem}\label{rem: properties of reddening sequences} Theorem \ref{thm: tame quivers have finitely many r-red sequences} follows from the following properties of reddening sequences.
\begin{enumerate}

\item (Finiteness) $\forall r,k$ only finitely many $r$-reddening sequences are disjoint from $\cV_k\backslash\cW_k$.
\item (Disjointness) $\forall r\ge0$ $\exists \kay_r$ so that every $r$-reddening sequence is disjoint from $\cV_{\kay_r}\backslash\cW_{\kay_r}$.
\end{enumerate}
These properties are proved in Propositions \ref{prop: only finitely many green seq of Class 1} and \ref{second proposition} below.
\end{rem}

\subsection{Finiteness}

In this subsection we will show that $R(T)\subseteq \cV_k\backslash\cW_k$ for all but finitely many clusters $T$. The first property in Remark \ref{rem: properties of reddening sequences} will follow.

\begin{lem}\label{lem: only finitely many roots outside V-W}
\emph{(a)} For each $k\ge1$ there are only finitely many real Schur roots $\gamma$ in the closure of the complement of $\cV_k\backslash\cW_k$.

\emph{(b)} $\cV_k\backslash\cW_k$ contains $R(T)$ for all but finitely many cluster tilting objects $T$.
\end{lem}

\begin{proof}
(a) If $\gamma$ is any preprojective root which is not in $\cP_{k+1}$ then $\Hom(M_\gamma,M_\alpha)=0$, so $\brk{\gamma,\alpha}\le0$, for any $\alpha\in\cP_{k+1}$ and $\Ext(M_\gamma,M_\alpha)\neq0$, and thus $\brk{\gamma,\alpha}<0$, for some $\alpha\in\cP_k$. Also, $\Ext(M_\gamma,M_\beta)=0$, so $\brk{\gamma,\beta}\ge0$, for any preinjective $\beta$ and $\Hom(M_\gamma,M_\beta)\neq0$, so $\brk{\gamma,\beta}>0$, for some preinjective $\beta$. Thus, $\gamma$ lies in the interior of $\cV_k\backslash\cW_k$.

Similarly, any preinjective $\gamma$ not in $\cI_k$ lies in the interior of $\cV_k\backslash \cW_k$. So, any real Schur root disjoint from the interior of $\cV_k\backslash\cW_k$ either lies in the finite set $\cP_{k+1}\cup \cI_{k+1}$ or is regular. Since there are only finitely many regular roots in the tame case, statement (a) follows.

(b) The dimension vector of every component $T_i$ of every cluster tilting object $T$ is a real Schur root. And $R(T)$ is spanned by the vectors $\undim T_i$. It follows from Proposition \ref{prop: R(T) does not cross dVk, dWk} that, if $R(T)$ is not contained in $\cV_k\backslash\cW_k$ then the interior of $R(T)$ is disjoint from $\cV_k\backslash\cW_k$. This implies that each $\undim T_i$ is a real Schur root in the closure of the complement of $\cV_k\backslash\cW_k$. By (a) there are only finitely many such roots. So, there are only finitely many $T$ outside of $\cV_k\backslash\cW_k$ and the remaining ones are inside $\cV_k\backslash\cW_k$.
\end{proof}

In the following lemma we say that two $c$-matrices are \emph{equivalent} if they differ by permutation of their columns, i.e., if they give the same set of $c$-vectors.

\begin{lem}\label{lemmaC} An $r$-reddening sequence passes through the same cluster at most $r+1$ times. In other words, no more than $r+1$ $c$-matrices in the sequence can be equivalent.
\end{lem}

\begin{proof} Suppose there is an $r$-reddening sequence $(k_0,\cdots,k_{m-1})$ which reaches the same $c$-matrix $C_s$ or an equivalent matrix say $q>r+1$ times. Apply the Rotation Lemma to make the first of these the first mutation so that $C_s$ is replaced with the identity matrix.

In the Mutation Formula \ref{formula}, if $C_{t}=C_sP_\rho$ then $C_{t}'=C_s'P_\rho$. Therefore, the $q-1$ $c$-matrices equivalent to $C_s$ in the original reddening sequence become $q-1$ permutation matrices in the rotated sequence. Since these have positive entries, the $q-1$ mutations preceding these must all be red. But $q-1>r$ giving a contradiction.
\end{proof}

These two lemmas imply the following.

\begin{prop}\label{prop: only finitely many green seq of Class 1}
For every $r,k$ there are at most finitely many $r$-reddening sequences disjoint from $\cV_k\backslash\cW_k$.\qed
\end{prop}

\subsection{Disjunction} We will show that all $r$-reddening sequences are disjoint from $\cV_k\backslash \cW_k$ for sufficiently large $k$. We use the fact that, in the tame case, there is a unique null root $\eta$ and $\tau\eta=\eta$. 
We also use the following formula from \cite{DR}.

\begin{thm}\label{thm: defect}
For any (connected) tame hereditary algebra there is a positive integer $m$ and, for every positive root $\alpha$ of $\Lambda$, there is an integer $\delta(\alpha)$ called the \emph{defect} of $\alpha$ so that
\[
	\tau^m\alpha= \alpha+\delta(\alpha)\eta.
\]
Furthermore, $\delta(\alpha)$ is positive, negative or zero depending on whether $\alpha$ is preinjective, preprojective or regular, respectively.
\end{thm}

Let $H(\eta)$ be the hyperplane in $\RR^n$ given by
\[
	H(\eta)=\{x\in\RR^n\,:\, \brk{x,\eta}=0\}, \text{ and let}
\]

\[
D(\eta):=\{x\in H(\eta)\,:\,\brk{x,\alpha}\le0\text{ for all preprojective roots }\alpha\}.
\]

We define the \emph{positive}, resp. \emph{negative}, side of $H(\eta)$ to be the set of all $x\in\RR^n$ so that $\brk{x,\eta}\ge0$, resp. $\le0$. All preprojective roots lie on the positive side of $H(\eta)$, preinjective roots and negative projective roots on the negative side and all regular roots lie on $H(\eta)$. Since every reddening sequence starts on the negative side of $H(\eta)$ and ends on its positive side, it must cross $H(\eta)$ at some point.

\begin{lem}
$\tau^{-1} D(\eta)=D(\eta)$.
\end{lem}

\begin{proof} $\tau^{-1} D(\eta)$ is the set of all $x\in H(\eta)$ so that $\left<\tau x,\alpha\right>\le0$ for all preprojective $\alpha$. Since $\left<\tau x,\alpha\right>=\left< x,\tau^{-1}\alpha\right>$, this condition is equivalent to the condition that $\left< x,\alpha\right>\le 0$ for $\alpha$ preprojective but not projective. So, $D(\eta)\subseteq \tau^{-1}D(\eta)$. But, for projective $\alpha$ and $x\in \tau^{-1}D(\eta)$, $\tau^{-m}\alpha$ is preprojective and $x\in H(\eta)$. So,
$
	\left< x,\alpha\right>=\left< x,\tau^{-m}\alpha-\delta(\tau^{-m}\alpha)\eta\right>=\left< x,\tau^{-m}\alpha\right>\le0$.
Therefore, $x\in D(\eta)$. 
\end{proof}

\begin{prop}\label{prop433}
Let $x\in H(\eta)$ and let $k\ge m$ where $m$ is as in Theorem \ref{thm: defect}. Then the following are equivalent.
\begin{enumerate}
\item $x\in D(\eta)$.
\item $\brk{x,\alpha}\le 0$ for all $\alpha\in\cP_k$.
\item $\brk{x,\beta}\ge0$ for all $\beta\in \cI_k$.
\item[($3'$)] $\brk{\tau^k x,\beta}\ge0$ for all $\beta\in \cI_k$.
\end{enumerate}
\end{prop}

\begin{proof} $(1) \then (2)$ by definition.\\
$(2)\ifff (3')$ since $-\tau^k(\cP_k)=\cI_k$.\\
$(3')\ifff (3)$ since $x\in D(\eta)$ iff $\tau^k x\in D(\eta)$.\\
$(2)\then(1)$ since, for any preprojective root $\alpha$ not in $\cP_k$, there is a positive integer $t$ so that $\tau^{tm}\alpha\in\cP_m\subseteq \cP_k$. Since $x\in H(\eta)$, $
	\brk{x,\alpha}=\brk{x,\tau^{tm}\alpha-t\delta(\alpha)\eta}=\brk{x,\tau^{tm}\alpha}\le 0
$
by (2). 
\end{proof}

\begin{cor}\label{cor: V cap H=D}
If $k\ge m$ then
\begin{enumerate}
\item[(a)] $\cV_k\cap H(\eta)=D(\eta)$.
\item[(b)] $\cW_k\cap H(\eta)=H(\eta)\backslash D(\eta)$.
\end{enumerate}
\end{cor}

\begin{proof} (a) follows from the equivalence $(1)\ifff(2)$ in Proposition \ref{prop433} and (b) follows from the equivalence $(1)\ifff(3)$ in the Proposition.\end{proof}

\begin{lem}\label{lem1}
For every preprojective or preinjective root $\gamma$, there is a $k$ so that $\gamma\notin\cV_k\backslash\cW_k$.
\end{lem}

\begin{proof}
Any preprojective $\gamma$ lies in $\cP_p$ for some $p$. Then $\brk{\gamma,\gamma}>0$ and $\gamma\in\cW_p$. So $\gamma\notin \cV_k\backslash \cW_k$ for all $k\ge p$. Similarly, any preinjective $\gamma$ lies in $\cI_q$ for some $q$. Then $\tau\gamma\in\cI_{q+1}$ and $\brk{\gamma,\tau\gamma}<0$. So, $\gamma\notin\cV_k\backslash\cW_k$ for all $k>q$.
\end{proof}

\begin{lem}\label{lem2}
Every cluster tilting object in the cluster category of $mod\text-\Lambda$ has at least one preprojective or preinjective summand.
\end{lem}

\begin{proof}
The dimension vectors of the summands of any cluster tilting object are linearly independent. But regular roots all lie in the hyperplane $H(\eta)$. So, the summands of a cluster tilting object cannot all be regular.
\end{proof}

\begin{lem}\label{lem3}
Every reddening sequence is disjoint from $\cV_k\backslash \cW_k$ for sufficiently large $k$.
\end{lem}

\begin{proof}
A reddening sequence consists of a finite sequence of cluster tilting objects each having at least one preprojective or preinjective summand. By Lemma \ref{lem1}, there is a $k$ so that none of these roots lies in $\cV_k\backslash\cW_k$. Then the reddening sequence stays in the complement of $\cV_k\backslash\cW_k$.
\end{proof}

\begin{prop}\label{second proposition}
If a reddening sequence meets $\cV_{rm}\backslash \cW_{rm}$ then it has at least $r$ red mutations. So, every $r$-reddening sequence is disjoint from $\cV_{(r+1)m}\backslash \cW_{(r+1)m}$.
\end{prop}

\begin{proof} We divide the proof into two cases and prove each case by induction on $r\ge1$, the case $r=0$ being vacuously true.

Suppose that a reddening sequence meets $\cV_{rm}\backslash \cW_{rm}$ for some $r\ge1$. Then the reddening sequence includes a cluster tilting object $T^1$ so that $R(T^1)\subseteq \cV_{rm}\backslash\cW_{rm}$. By Lemma \ref{lem3}, $R(T^1)$ is disjoint from $\cV_k\backslash\cW_k$ for $k$ sufficiently large. Since $\cV_k\backslash \cW_k$ contains $D(\eta)=H(\eta)\cap(\cV_{rm}\backslash \cW_{rm})$, the region $R(T^1)$ lies on one side of the hyperplane $H(\eta)$. There are two cases. Either $R(T^1)$ lies on the negative side of $H(\eta)$ or it lies on its positive side.

\underline{Case 1}: $R(T^1)$ lies on the negative side of $H(\eta)$. 

In this case we will show, by induction on $r$, that the remainder of the reddening sequence has at least $r$ red mutations. 

Since $R(T^1)$ is on the negative side of $H(\eta)$, the remainder of the reddening sequence must somehow arrive at the positive side of $H(\eta)$. By Lemma \ref{lem3} the sequence is disjoint from some $\cV_k\backslash \cW_k$ which contains $D(\eta)$ by Corollary \ref{cor: V cap H=D}. So, the reddening sequence must pass through $H(\eta)\backslash D(\eta)$ which is in $\cW_{rm}\backslash \cV_{rm}$. To get from $\cV_{rm}\backslash \cW_{rm}$ to $\cW_{rm}\backslash \cV_{rm}$, the reddening sequence must pass through one of the red walls $D(\beta)$ of $\cV_{rm}$ on the negative side of $H(\eta)$. Let $T^2,T^3$ be the two cluster tilting objects in the reddening sequence with $R(T^2)\subseteq \cV_{rm}$, $R(T^3)\not\subseteq \cV_{rm}$ and $R(T^2)\cap R(T^3)\subseteq R(\beta)\cap \partial \cV_{rm}$.

Let $x\in D(\beta)\cap \partial \cV_{rm}$ be a point in the interior of the wall separating $R(T^2)$ and $R(T^3)$. Then $\brk{x,\beta}=0$ and $\beta\in\cI_{rm}$. We claim that $\beta$ does not lie in $\cI_{(r-1)m}$. Otherwise, $\tau^m\beta =\beta+\delta(\beta)\eta$ would lie in $\cI_{rm}$ and we would arrive at the contradiction
\[
	0\le\brk{x,\tau^m\beta}=\brk{x,\beta}+\delta(\beta)\brk{x,\eta}<0
\]
using the fact that $\delta(\beta)>0$ for preinjective $\beta$ and $\brk{x,\eta}<0$ in Case 1. 

But $x\in\cV_{rm}\backslash\cW_{rm}\subseteq \cV_{(r-1)m}\backslash\cW_{(r-1)m}$. Since $\beta\notin\cI_{(r-1)m}$, $x$ does not lie on $\partial \cV_{(r-1)m}$. So, $x$ lies in the interior of $\cV_{(r-1)m}$. This implies that $R(T^3)$ also lies in the interior of $\cV_{(r-1)m}\backslash \cW_{(r-1)m}$ and on the negative side of $D(\eta)$. By induction on $r$, the rest of the redding sequence has at least $r-1$ red mutations. Since the mutation from $T^2$ to $T^3$ is red, the portion of the reddening sequence after $T^1$ has at least $r$ red mutations. This proves the proposition in Case 1.

\underline{Case 2}: $R(T^1)$ lies on the positive side of $H(\eta)$.

In this case we claim that the part of the reddening sequence before $T^1$ has at least $r$ red mutations. By an argument analogous to Case 1, there is a $T^0$ with $R(T^0)$ in $\cV_{(r-1)m}\backslash \cW_{(r-1)m}$ in the reddening sequence. We need at least one red mutation to get from $T^0$ to $T^1$ and, by induction on $r$, we need $r-1$ red mutations to get to $T^0$. This gives at least $r$ red mutations in Case 2, just as in Case 1.

So, every reddening sequence which meets $\cV_{rm}\backslash\cW_{rm}$ has at least $r$ red mutations.
\end{proof}

\begin{center}
\begin{figure}[h]

\includegraphics{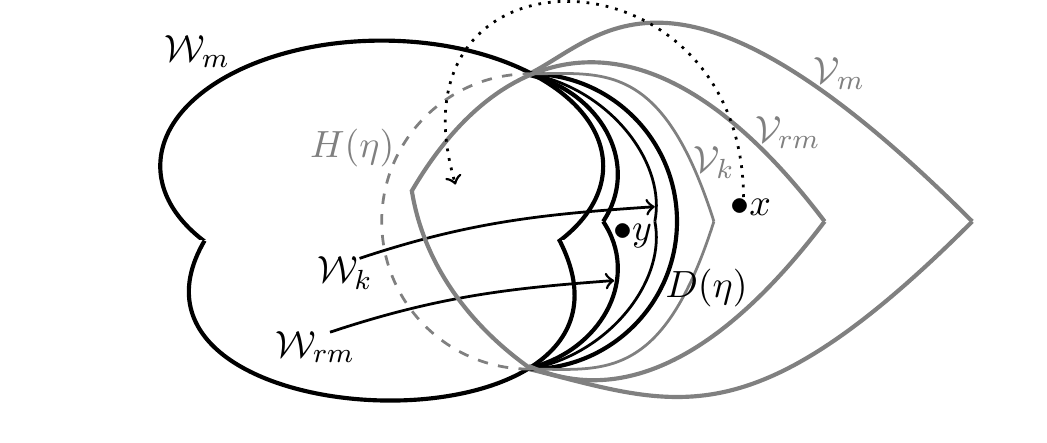}

\caption{(Proof of Proposition \ref{second proposition}) Case 1: Since reddening sequences cannot cross $D(\eta)$, we need to cross $r$ red walls (gray in figure) to escape from any interior point $x\in \cV_{rm}$ on the negative side of $D(\eta)$ as shown in the dotted path. Case 2: We need to cross $r$ red walls (black in figure) to reach any interior point $y\in \cV_{rm}\backslash\cW_{rm}$ on the positive side of $D(\eta)$.}\label{fig: need r red}\label{fig06}
\end{figure}
\end{center}

\begin{thm}\label{finiteness}
Let $Q$ a valued quiver which is mutation equivalent to an acyclic valued quiver of tame representation type. Then, for any $r\ge0$, $Q$ admits only finitely many $r$-reddening sequences. In particular, $Q$ has only finitely many maximal green sequences.
\end{thm}

\begin{proof} Suppose first that $Q$ is acyclic. By Proposition \ref{second proposition}, every $r$-reddening sequence is disjoint from $\cV_{(r+1)m}\backslash\cW_{(r+1)m}$. By Proposition \ref{prop: only finitely many green seq of Class 1} there are only finitely many such sequences. Therefore, there are only finitely many $r$-reddening sequences when $Q$ is acyclic with tame representation type.

In the case when $Q$ is not acyclic, take a fixed mutation sequence $(j_1,\cdots,j_t)$ so that $Q'=\mu_{j_t}\cdots\mu_{j_1}Q$ is a tame acyclic quiver. Every $r$-reddening sequence $(k_0,\cdots,k_s)$ for $Q$ gives an $(r+t)$-reddening sequence $(j_1,\cdots,j_t,j_t,\cdots,j_1,k_0,\cdots,k_s)$ for the same quiver $Q$. Let $\sigma$ be the permutation associated to this sequence. By the Rotation Lemma, $(j_t,\cdots,j_t,$ $k_0,\cdots,k_s,$ $\sigma^{-1}(j_1),\cdots,\sigma^{-1}(j_t))$ is an $(r+t)$-reddening sequence for $Q'$. Since $Q'$ is acyclic, there are only finitely many such sequences. Therefore, there are only finitely many possibilities for the middle part of the sequence which is an arbitrary $r$-reddening sequence for $Q$.
\end{proof}

\section*{Acknowledgements}

This paper is a report on a joint project initiated during the Hall and Cluster Algebras Conference at Centre de Recherches Math\'ematiques, University of Montreal, May 8-12, 2014. The authors would like to thank CRM for hosting this very productive event. We also thank Milen Yakimov whose lecture on maximal green sequences at that event inspired the conversations between the authors.

The first author is supported by NSERC and Bishop's University. The third author acknowledges supported of the National Security Agency, the fourth author was supported by the National Science Foundation. Also, the authors had very useful inspiring conversations with Nathan Reading, Al Garver, Greg Muller and Milen Yakimov at various events, especially the Conference on Strings, Quivers and Cluster Algebras in Mathematical Physics at the Korean Institute for Advanced Study (KIAS) in Seoul, Korea, Dec. 18-22, 2014. We wish to thank KIAS and Kyungyong Lee and the other organizers of this very enjoyable and fruitful conference. We also thank Bishop's University and the organizers of the XXVII-th meeting on Representation Theory of Algebras, Sept 4-5, 2015 where the four authors had a chance to meet to finish this project.


\begin{thebibliography}{aa}

\bibitem{ACCERV} Murad Alim, Sergio Cecotti, Clay C\'ordova, Sam Espahbodi, Ashwin Rastogi, and Cumrun Vafa, \emph{BPS quivers and spectra of complete N = 2 quantum field theories}, Comm. Math. Phys., 323(3):1185--1227, 2013.

\bibitem{BDP}
Thomas Br\"{u}stle, Gr\'egoire Dupont, and  Matthieu P\'erotin, \emph{On maximal green sequences}, Int Math Res Notices (2014),  4547--4586.

\bibitem{BY} Thomas Br\"ustle and  Dong Yang, {\em On Ordered Exchange Graphs, } EMS Series of Congress Reports: Advances in Representation Theory of Algebras (2014), 135--193.

\bibitem{BMRRT} Aslak~Bakke Buan, Robert~J. Marsh, Idun Reiten, and Gordana Todorov, \emph{Tilting theory and cluster combinatorics}, Adv. Math. \textbf{204}
  (2006), no.~2, 572--618.

\bibitem{DWZ}
Harm Derksen, Jerzy Weyman, and Andrei Zelevinsky, \emph{Quivers with potentials and their representations II: Applications to cluster algebras}, J. Amer. Math. Soc., 23:749--790, 2010.

\bibitem{DR}
Vlastimil Dlab and Claus~Michael Ringel, \emph{Indecomposable representations
  of graphs and algebras}, Mem. Amer. Math. Soc. \textbf{6} (1976), no.~173,
  v+57.
  
\bibitem{FZ} 
Sergey Fomin and Andrei Zelevinsky, \emph{Cluster algebras. {IV}. {C}oefficients}, Compos. Math. \textbf{143} (2007), no.~1, 112--164.


\bibitem{GHKK} Mark Gross, Paul Hacking, Sean Keel, and Maxim Kontsevich, \emph{Canonical basis for cluster algebras}, arXiv:1411.1394.


\bibitem{Keller} Bernhard Keller, \emph{Quiver mutation and combinatorial DT-invariants},  The 25th International Conference on Formal Power Series and Algebraic Combinatorics. Paris (2013). DMTCS proc. AS, 2013, 9--20.


\bibitem{IOTW1}Kiyoshi Igusa, Kent Orr, Gordana Todorov, and Jerzy Weyman, \emph{Cluster complexes via semi-invariants}, Compos. Math. \textbf{145} (2009), no.~4, 1001--1034.


\bibitem{IOTW2} Kiyoshi Igusa, Kent Orr, Gordana Todorov, and Jerzy Weyman, \emph{Modulated semi-invariants}, Preprint arXiv:1507.03051v2.


\bibitem{IT13} Kiyoshi Igusa and Gordana Todorov, \emph{Signed exceptional sequences and the cluster morphism category}, preprint 2014.


\bibitem{IT14} Kiyoshi Igusa and Gordana Todorov, \emph{Picture groups and maximal green sequences}, preprint 2014.


\bibitem{M} Greg Muller, \emph{The existence of a maximal green sequence is not invariant under quiver mutation}, Preprint arXiv:1503.04675.


\bibitem{NZ} Tomoki Nakanishi, Andrei Zelevinsky, \emph{On tropical dualities in cluster algebras}, Contemp. Math. 565 (2012) 217--226.


\bibitem{R13} Nathan Reading, \emph{Universal geometric cluster algebras}, arXiv:1411.1394.

\bibitem{N} Bertrand Nguefack, \emph{Potentials and Jacobian algebras for tensor algebras of bimodules}, arXiv:1004.2213.


\bibitem{X} Dan Xie, \emph{BPS spectrum, wall crossing and quantum dilogarithm identity}, Preprint arXiv:1211:7071.


\end{thebibliography}
\end{document}